\documentclass[a4paper]{amsart}
\usepackage{wrapfig}
\usepackage[dvips]{graphicx}
\usepackage{amsmath,amsthm,amssymb,amscd}
\usepackage{mathrsfs}
\usepackage[all]{xy}
\theoremstyle{definition}
\newtheorem{thm}{Theorem}[section]

\newtheorem{pro}[thm]{Proposition}
\newtheorem{cor}[thm]{Corollary}
\newtheorem{lem}[thm]{Lemma}

\newtheorem{rem}[thm]{Remark}

\newtheorem*{mainthm}{Theorem A}
\newtheorem*{mainthm2}{Theorem B}
\theoremstyle{definition}

\begin{document}
\title[Equivariant absorption for the Razak-Jacelon algebra]
{Equivariant Kirchberg-Phillips type absorption for the Razak-Jacelon algebra}
\author{Norio Nawata}
\address{Department of Pure and Applied Mathematics, Graduate School of Information Science and Technology, Osaka University, Yamadaoka 1-5, Suita, Osaka 565-0871, Japan}
\email{nawata@ist.osaka-u.ac.jp}
\thanks{This work was supported by JSPS KAKENHI Grant Number 20K03630}

\begin{abstract}
Let $A$ and $B$ be simple separable nuclear monotracial C$^*$-algebras, and let $\alpha$ and $\beta$ 
be strongly outer actions of a countable discrete amenable group $\Gamma$ on 
$A$ and $B$, respectively. In this paper, we show that $\alpha\otimes\mathrm{id}_{\mathcal{W}}$ 
on $A\otimes\mathcal{W}$ and 
$\beta\otimes\mathrm{id}_{\mathcal{W}}$ on $B\otimes\mathcal{W}$ are cocycle conjugate where 
$\mathcal{W}$ is the Razak-Jacelon algebra. 
Also, we characterize such actions by using 
the fixed point subalgebras of Kirchberg's central sequence C$^*$-algebras.
\end{abstract}
\maketitle

\section{Introduction}

The Razak-Jacelon algebra $\mathcal{W}$ is a simple separable nuclear monotracial 
$\mathcal{Z}$-stable C$^*$-algebra which is $KK$-equivalent to $\{0\}$. 
By classification results \cite{CE} and \cite{EGLN} (see also \cite{Na4}), 
we see that such a C$^*$-algebra is unique. 
The Razak-Jacelon algebra $\mathcal{W}$ is an important example of a simple separable nuclear 
stably projectionless C$^*$-algebra. 
We refer the reader to \cite{EGLN0}, \cite{EGLN}, \cite{GL2}, \cite{GL3} 
and \cite{L} for recent progress of the classification of simple separable nuclear 
stably projectionless C$^*$-algebras. 
See, for example, \cite{Ell}, \cite{J},  \cite{KK1}, \cite{Raz} and \cite{Tsa} for concrete constructions of 
$\mathcal{W}$ and simple stably projectionless C$^*$-algebras. 
We can regard $\mathcal{W}$ as a stably finite analog of the 
Cuntz algebra $\mathcal{O}_2$ generated by two isometries. 
Kirchberg and Phillips showed in \cite{KP} that a simple separable nuclear unital C$^*$-algebra $B$ is 
isomorphic to $\mathcal{O}_2$ if and only if there exists an asymptotically central 
unital homomorphism from $\mathcal{O}_2$ to $B$, that is, there exists a unital homomorphism 
from $\mathcal{O}_2$ to the central sequence C$^*$-algebra $B^{\omega}\cap B^{\prime}$ of $B$. 
In particular, if $A$ is a simple separable nuclear unital C$^*$-algebra, then 
$A\otimes\mathcal{O}_2$ is isomorphic to $\mathcal{O}_2$. 
As an analog of this result, 
we see that if $A$ is a simple separable nuclear monotracial C$^*$-algebra, then 
$A\otimes\mathcal{W}$ is isomorphic to $\mathcal{W}$ 
by classification results \cite{CE} and \cite{EGLN}. 
In \cite{Na4}, the author characterized $\mathcal{W}$ by using properties of 
Kirchberg's central sequence C$^*$-algebra $F(\mathcal{W})$ of $\mathcal{W}$. 
Indeed, he showed that a simple separable nuclear monotracial C$^*$-algebra $B$ is isomorphic to 
$\mathcal{W}$ if and only if $B$ satisfies the following properties: \ \\
\ \\
(i) for any $\theta\in [0,1]$, there exists a projection $p$ in $F(B)$ such that 
$\tau_{B, \omega}(p)=\theta$, \ \\
(ii) if $p$ and $q$ are projections in $F(B)$ such that $0<\tau_{B, \omega}(p)=\tau_{B, \omega}(q)$, 
then  $p$ is Murray-von Neumann equivalent to $q$,  \ \\
(iii) there exists an injective homomorphism from $B$ to $\mathcal{W}$, \ \\
\ \\
where $\tau_{B,\omega}$ is the induced tracial state on $F(B)$ by the unique tracial state on $B$. 

The classification of group actions on operator algebras is 
one of the central themes in the theory of operator algebras. 
We refer the reader to \cite{I}, \cite{MT}, \cite{Sza} and the references given 
there for details and previous results. 
We shall recall only some results that are directly related to this paper. 
In \cite{I1}, Izumi showed that an outer action $\beta$ of a finite group on $\mathcal{O}_2$ has the 
Rohlin property if and only if there exists a unital homomorphism from $\mathcal{O}_2$ to 
the fixed point subalgebra $(\mathcal{O}_2^{\omega}\cap\mathcal{O}_2^{\prime})^{\beta}$. 
In particular, if $\alpha$ is an outer action of a finite group on a simple separable nuclear 
unital C$^*$-algebra, then $\alpha\otimes\mathrm{id}_{\mathcal{O}_2}$ on $A\otimes\mathcal{O}_2$ 
has the Rohlin property. 
Since Rohlin actions of a finite group on $\mathcal{O}_2$ are unique up to conjugacy, 
we can regard this result as an equivariant version of the Kirchberg-Phillips type absorption 
for outer actions of finite groups. 
As an analog of this result, the author showed that if $\alpha$ is a strongly outer action of a 
finite group on a simple separable nuclear monotracial C$^*$-algebra, then 
$\alpha\otimes\mathrm{id}_{\mathcal{W}}$ on $A\otimes\mathcal{W}$ has the 
Rohlin property in \cite{Na3}. 
In \cite{Sza4}, Szab\'o generalized Izumi's result to actions of 
countable discrete amenable groups. 
For any countable discrete amenable group $\Gamma$, 
he gave the model action $\delta^{\Gamma}$ of $\Gamma$ on $\mathcal{O}_2$ 
such that if $\alpha$ is an action of $\Gamma$ on a unital Kirchberg algebra $A$, 
then $\alpha\otimes \delta^{\Gamma}$ on $A\otimes\mathcal{O}_2$ is strongly cocycle conjugate 
to $\delta^{\Gamma}$. 
Also, he showed that if $\alpha$ is an outer action of $\Gamma$ 
on a simple separable nuclear unital C$^*$-algebra $A$, then 
$\alpha\otimes\mathrm{id}_{\mathcal{O}_2}$ is strongly cocycle conjugate to $\delta^{\Gamma}$. 
Note that  Szab\'o's results are based on strongly self-absorbing actions in 
\cite{Sza3}, \cite{Sza1}, \cite{Sza1c} and \cite{Sza2}. 
Moreover, Suzuki generalized Szab\'o's  result to amenable actions (or actions with the quasi-central 
approximation property (QAP)) of countable discrete exact groups in \cite{Su}. 

In this paper, we consider an equivariant version of the Kirchberg-Phillips type absorption 
for strongly outer actions of countable discrete amenable groups on $\mathcal{W}$. 
In particular, one of the main results is the following theorem. 

\begin{mainthm} (Corollary \ref{cor:main1}) \ \\
Let $A$ and $B$ be simple separable nuclear monotracial C$^*$-algebras, and let $\alpha$ and $\beta$ 
be strongly outer actions of a countable discrete amenable group $\Gamma$ on 
$A$ and $B$, respectively. Then $\alpha\otimes\mathrm{id}_{\mathcal{W}}$ on $A\otimes\mathcal{W}$ 
is cocycle conjugate to $\beta\otimes\mathrm{id}_{\mathcal{W}}$ on $B\otimes\mathcal{W}$.
\end{mainthm}

Note that we cannot expect an analog of Suzuki's generalization in \cite{Su} for $\mathcal{W}$
because no action of a countable discrete non-amenable exact group 
on $\mathcal{W}$ is amenable (or QAP). Also, we must assume that actions are 
strongly outer. 

For any countable discrete (amenable) group $\Gamma$, there exists a strongly outer action 
of $\Gamma$ on the uniformly hyperfinite (UHF) algebra $M_{2^{\infty}}$ 
of type $2^{\infty}$. Indeed, the Bernoulli shift action $\mu^{\Gamma}$ on 
$\bigotimes_{g\in \Gamma}M_{2^{\infty}}\cong M_{2^{\infty}}$ is such an action. 
Therefore we can regard $\mu^{\Gamma}\otimes\mathrm{id}_{\mathcal{W}}$ as a model action 
on $M_{2^{\infty}}\otimes\mathcal{W}\cong \mathcal{W}$. 
We obtain the theorem above as a corollary of 
the following characterization of such actions by using the fixed point subalgebras of Kirchberg's 
central sequence C$^*$-algebras, which is also one of the main results. 

\begin{mainthm2} (Theorem \ref{thm:main2}) \ \\
Let $\gamma$ be a strongly outer action of a countable discrete amenable group $\Gamma$ on 
$\mathcal{W}$. Then $\gamma$ is cocycle conjugate to 
$\mu^{\Gamma}\otimes\mathrm{id}_{\mathcal{W}}$ on $M_{2^{\infty}}\otimes \mathcal{W}$ 
if and only if $(\mathcal{W}, \gamma)$ satisfies the following properties: \ \\
\ \\
(i) for any $\theta\in [0,1]$, there exists a projection $p$ in 
$F(\mathcal{W})^{\gamma}$ such that $\tau_{\mathcal{W}, \omega}(p)=\theta$, \ \\
(ii) if $p$ and $q$ are projections in $F(\mathcal{W})^{\gamma}$ 
such that $0<\tau_{\mathcal{W}, \omega}(p)=\tau_{\mathcal{W}, \omega}(q)$, 
then  $p$ is Murray-von Neumann equivalent to $q$, \ \\
(iii) There exists an injective homomorphism from 
$\mathcal{W}\rtimes_{\gamma}\Gamma$ to $\mathcal{W}$.
\end{mainthm2}

Our main technical tool of the proof of the theorem above is Szab\'o's approximate cocycle 
intertwining argument in \cite{Sza7} (see also \cite{Ell2}). 
This is an equivariant version of Elliott's approximate intertwining argument \cite{Ell0}.
For this argument, we need to show a uniqueness type theorem and an existence type theorem in 
a suitable sense. 
Using properties of fixed point subalgebras of usual central sequence C$^*$-algebras
and Kirchberg's relative central sequence C$^*$-algebras, 
we obtain these theorems. 
Note that our analysis of fixed point subalgebras of usual central sequence C$^*$-algebras
and Kirchberg's relative central sequence C$^*$-algebras 
are essentially based on Sato's results \cite{Sa2}, Szab\'o's results \cite{Sza6} 
and Matui-Sato's results \cite{MS}, \cite{MS2}, \cite{MS3}. 
In particular, techniques around property (SI) is important. 
See also \cite{Sa0}, \cite{Sa} for the pioneering works and \cite{BBSTWW} for generalizations. 

This paper is organized as follows. In Section \ref{sec:pre}, we collect notations and definitions.
In Section \ref{sec:fixed}, we shall consider properties of fixed point subalgebras of usual 
central sequence C$^*$-algebras and Kirchberg's relative central sequence C$^*$-algebras
in general settings. As explained above, arguments in this section are essentially based on 
\cite{Sa2}, \cite{Sza6}, \cite{MS}, \cite{MS2} and \cite{MS3}. Also,  
Ocneanu's results in \cite{Oc} play an important role. 
In Section \ref{sec:properties}, we shall consider properties of the fixed point subalgebra 
$F(A\otimes\mathcal{W})^{\alpha\otimes\mathrm{id}_{\mathcal{W}}}$ where $\alpha$ is 
an outer action of a countable discrete amenable group on 
a simple separable nuclear monotracial C$^*$-algebra $A$. 
This section is based on the author's previous works in \cite{Na2} and \cite{Na3}. 
In Section \ref{sec:sza}, we shall consider a slight variant of Szab\'o's approximate cocycle 
intertwining argument, which is our main technical tool in this paper. 
In Section \ref{sec:unique}, we shall show a uniqueness type theorem. Our proof of this theorem 
is based on Connes' $2\times 2$ matrix trick and properties of Kirchberg's relative central sequence 
C$^*$-algebras. 
In Section \ref{sec:exist}, we shall show an existence type theorem.  
Many arguments in this section are based on \cite[Section 5]{Na4} and Schafhauser's ideas \cite{Sc} 
in his proof of the Tikuisis-White-Winter theorem \cite{TWW}. Also, we use properties of fixed point 
subalgebras of usual central sequence C$^*$-algebras. 
In Section \ref{sec:main}, we shall show the main results in this paper.

\section{Preliminaries}\label{sec:pre}

In this section we shall collect notations and definitions.  
We refer the reader to \cite{Bla} and \cite{Ped2} for basics of operator algebras. 

For a C$^*$-algebra $A$, we denote by $A_{+}$ the set of positive elements of $A$,  
by $A^{\sim}$ the unitization algebra of $A$ and by $M(A)$ the multiplier algebra of $A$. 
Note that if $A$ is unital, then $A=A^{\sim}=M(A)$. Let $\mathrm{GL}(A^{\sim})$ be the set of 
invertible elements in $A^{\sim}$. 
For a unitary element $u$ in $M(A)$, define an automorphism $\mathrm{Ad}(u)$ of $A$ by 
$\mathrm{Ad}(u)(a)=uau^*$. Such an automorphism is said to be an \textit{inner automorphism} of 
$A$. 
Every automorphism $\alpha$ of $A$ induces an automorphism of $M(A)$. 
We denote it by the same symbol $\alpha$ for simplicity. 
For $a,b\in A$, let $[a,b]$ be the commutator $ab-ba$. 
We say that $A$ is \textit{monotracial} if $A$ has a unique tracial state 
and no unbounded traces. In the case where $A$ is monotracial, we denote 
by $\tau_{A}$ the unique tracial state on $A$ unless otherwise specified. 
In general, every tracial state $\tau$ on $A$ extends uniquely to a tracial state on $M(A)$. 
We denote it by the same symbol $\tau$ for simplicity. 
We say that positive elements $a$ and $b$ are \textit{Murray-von Neumann equivalent in $A$} if 
there exists an element $z$ in $A$ such that $z^*z=a$ and $zz^*=b$. 
For any $a\in A_{+}$ and $\varepsilon>0$, let $(a-\varepsilon)_{+}:=f_{\varepsilon}(a)\in A_{+}$ where 
$f_{\varepsilon}(t)= \max\{0, t-\varepsilon\}$, $t\in\mathrm{Sp}(a)$.

\subsection{Group actions}
In this paper, we assume that $\Gamma$ is a countable discrete group 
unless otherwise specified. 
An \textit{action} $\alpha$ of $\Gamma$ on a C$^*$-algebra $A$ is a 
homomorphism from $\Gamma$ to $\mathrm{Aut}(A)$. 
We denote by $A\rtimes_{\alpha}\Gamma$ and $A^{\alpha}$ the reduced crossed product C$^*$-algebra 
and the fixed point subalgebra, respectively.  
Since $\Gamma$ is discrete, $\alpha$ induces an action on $M(A)$. We denote it by the same symbol 
$\alpha$ for simplicity. 
An $\alpha$-\textit{cocycle} $u$ is a map from 
$\Gamma$ to the unitary group of $M(A)$ such that $u_{gh}=u_g\alpha_g(u_{h})$ 
for any $g, h\in\Gamma$. We say that an $\alpha$-cocycle $u$ is a \textit{coboundary} if there exists a 
unitary element $w$ in $M(A)$ such that $u_g=w\alpha_g(w^*)$ for any $g\in\Gamma$. 
Let $\alpha$ and $\beta$ be actions of $\Gamma$ on $A$ and $B$, respectively. 
We say that $\alpha$ and $\beta$ are \textit{conjugate} if there exist an isomorphism 
$\varphi$ from $A$ onto $B$ such that $\varphi \circ \alpha_{g} =\beta_g \circ \varphi$ 
for any $g\in \Gamma$. They are said to be \textit{cocycle conjugate} 
if there exist an isomorphism $\varphi$ from $A$ onto $B$ and $\beta$-cocycle $u$ such that 
$\varphi \circ \alpha_{g} = \mathrm{Ad}(u_g) \circ \beta_g \circ \varphi$ for any $g\in \Gamma$. 
If $\alpha$ and $\beta$ are cocycle conjugate, then $A\rtimes_{\alpha}\Gamma$ is isomorphic to 
$B\rtimes_{\beta}\Gamma$, but $A^{\alpha}$ is not isomorphic to $B^{\beta}$ in general.  
On the other hand, if 
$\alpha$ and $\beta$ are conjugate, then $A\rtimes_{\alpha}\Gamma$ and $A^{\alpha}$ are 
isomorphic to $B\rtimes_{\beta}\Gamma$ and $B^{\beta}$, respectively. 

An action $\alpha$ of $\Gamma$ on $A$ is said to be \textit{outer} if $\alpha_g$ is not an inner 
automorphism of $A$ 
for any $g\in \Gamma\setminus \{\iota\}$ where $\iota$ is the identity of $\Gamma$. 
Assume that $A$ is monotracial, and consider the Gelfand-Naimark-Segal (GNS) representation 
$(\pi_{\tau_A}, H_{\tau_A})$ of $\tau_{A}$. Then $\alpha$ induces an action $\tilde{\alpha}$ of $\Gamma$ 
on $\pi_{\tau_{A}}(A)''$. We say that $\alpha$ is \textit{strongly outer} if 
$\tilde{\alpha}_g$ is not an inner automorphism of $\pi_{\tau_{A}}(A)''$ for any 
$g\in \Gamma\setminus \{\iota\}$. 
(See, for example, \cite{GH} for the definition of strongly outerness in more general settings.) 
It is known that if $\alpha$ is a strongly outer action of $\Gamma$ on a simple monotracial 
C$^*$-algebra $A$, then $A\rtimes_{\alpha}\Gamma$ is simple and monotracial.  
In particular, the unique tracial state $\tau_{A\rtimes_{\alpha}\Gamma}$ is given by 
$\tau_{A}\circ E_{\alpha}$ where $E_{\alpha}$ is the canonical conditional expectation 
from $A\rtimes_{\alpha}\Gamma$ onto $A$. 

\subsection{Kirchberg's relative central sequence C$^*$-algebras}

Fix a free ultrafilter $\omega$ on $\mathbb{N}$. 
For a C$^*$-algebra $B$, put 
$$
B^{\omega}:= \ell^{\infty}(\mathbb{N}, B)/\{\{x_n\}_{n\in\mathbb{N}}\in \ell^{\infty}(\mathbb{N}, B)\; |
\; \lim_{n\to\omega} \|x_n \|=0 \}.
$$
We denote by $(x_n)_n$ a representative of an element in $B^{\omega}$. 
Let $\Phi$ be a homomorphism from a C$^*$-algebra $A$ to $M(B)^{\omega}$. 
Set 
$$
B^{\omega}\cap \Phi (A)^{\prime}:= \{(x_n)_n\in B^{\omega}\; |\; [(x_n)_n, \Phi (a)]=0 \text{ for any } a\in A
\}
$$
and 
$$
\mathrm{Ann}(\Phi (A), B^{\omega}):=\{(x_n)_n\in B^{\omega}\cap \Phi (A)^{\prime}\; |\; 
(x_n)_n\Phi(a)=0 \text{ for any } a\in A\}. 
$$
Note that we do not assume that $\Phi(A)$ is a subalgebra of $B^{\omega}$ in general but 
$(x_n)_n\Phi (a)$ and $\Phi(a)(x_n)_n$ are elements in $B^{\omega}$ for any 
$(x_n)_n\in B^{\omega}$ and $a\in A$. 
It is easy to see that $\mathrm{Ann}(\Phi (A), B^{\omega})$ is a closed ideal of 
$B^{\omega}\cap \Phi (A)^{\prime}$. 
Define \textit{Kirchberg's relative central sequence C$^*$-algebra} $F(\Phi(A), B)$ 
of $\Phi$ by 
$$
F(\Phi(A), B):=B^{\omega}\cap \Phi(A)^{\prime}/\mathrm{Ann}(\Phi(A), B^{\omega}).
$$ 
We identify $B$ with the C$^*$-subalgebra of $B^\omega$ consisting of equivalence 
classes of constant sequences. In the case $A=B$ and $\Phi=\mathrm{id}_{B}$, 
we denote $F(\Phi(B), B)$ by $F(B)$ and call it \textit{Kirchberg's central sequence C$^*$-algebra}. 
If $A$ is $\sigma$-unital and $\Phi(A)\subseteq B^{\omega}$, then $F(\Phi(A), B)$ has a unit by 
\cite[Proposition 1.9]{Kir2} (see also \cite[Section 2.2]{Na4}). 

For a tracial state $\tau_{B}$ on $B$, 
define a tracial state $\tau_{B, \omega}$ on $M(B)^{\omega}$ by 
$\tau_{B, \omega}((x_n)_n):= \lim_{n\to\omega}\tau_{B}(x_n)$ for any $(x_n)_n\in M(B)^{\omega}$. 
If $\tau_{B, \omega}\circ \Phi$ is a state on $A$, then $\tau_{B, \omega}$ induces a tracial state on 
$F(\Phi(A), B)$ by the same proof as in \cite[Proposition 2.1]{Na4}. We denote it by the same symbol 
$\tau_{B, \omega}$ for simplicity. 

Let $\beta$ be an action of $\Gamma$ on $B$. Since $\Gamma$ is discrete, $
\beta$ induces an action of $\Gamma$ on 
$M(B)^{\omega}$ and $F(B)$, respectively. 
We denote them by the same symbol $\beta$ for simplicity. 
It is easy to see that if an action $\alpha$ on $A$ is cocycle conjugate to an action $\beta$ on $B$, 
then $\alpha$ on $F(A)$ is conjugate to $\beta$ on $F(B)$. 
In particular, $F(A)^{\alpha}$ is isomorphic to $F(B)^{\beta}$.  
We say that an action $\beta^{\omega}$ of $\Gamma$ on $M(B)^{\omega}$ (respectively,  
$F(\Phi(A)), B)$) is \textit{semiliftable} if there exists a set 
$\{\beta_{g,n}\; |\; g\in\Gamma, n\in\mathbb{N}\}$ of automorphisms of $B$ such that 
$\beta_g^{\omega}((x_n)_n)=(\beta_{g,n}(x_n))_n$ for any $g\in G$ and $(x_n)_n\in M(B)^{\omega}$
(respectively, $\beta^{\omega}([(x_n)_n])=[(\beta_{g,n}(x_n))_n]$ for any $g\in G$ and 
$[(x_n)_n]\in F(\Phi(A)), B)$.)
If a semiliftable action $\beta^{\omega}$ of $\Gamma$ on $M(B)^{\omega}$ satisfies 
$\beta_g( \Phi(A))= \Phi(A)$ for any $g\in \Gamma$, then $\beta^{\omega}$ 
induces (semiliftable) actions on $B^{\omega}\cap \Phi(A)^{\prime}$ and $F(\Phi(A), B)$, respectively.
We denote them by the same symbol $\beta^{\omega}$ for simplicity.

For analyzing (relative) central sequence algebras, it is useful to consider the reindexing argument and 
the diagonal argument (or Kirchberg's $\varepsilon$-test \cite[Lemma A.1]{Kir2}). 
We refer the reader to \cite[Section 1.3]{BBSTWW} and \cite[Chapter 5]{Oc} for details of these 
arguments. 

\section{Fixed point subalgebra of relative central sequence C$^*$-algebras}\label{sec:fixed}

In this section we shall consider properties of 
$(B^{\omega}\cap \Phi(A)^{\prime})^{\beta^{\omega}}$ and $F(\Phi (A), B)^{\beta^{\omega}}$ 
for certain $A$, $B$, $\Phi$ and $\beta^{\omega}$. 
Note that arguments in this section are essentially based on 
Sato's observation \cite{Sa2}, 
Szab\'o's arguments \cite{Sza6} and Matui-Sato's techniques \cite{MS}, \cite{MS2}, \cite{MS3}. 
Also, Ocneanu's results in \cite{Oc} play an important role. 

We say that a monotracial C$^*$-algebra $A$ has \textit{strict comparison} if for any $k\in\mathbb{N}$, 
$a,b \in (A\otimes M_{k}(\mathbb{C}))_{+}$ with 
$d_{\tau_A\otimes Tr_{k}}(a)<d_{\tau_A\otimes Tr_{k}}(b)$, there exists a sequence 
$\{r_n\}_{n\in\mathbb{N}}$ in $A\otimes M_{k}(\mathbb{C})$ such that $\| r_n^* br_n- a\|\to 0$. 
The following lemma is essentially based on \cite[Lemma 3.2]{Sza6}. 

\begin{lem}\label{lem:property-si}
Let $A$ be a simple separable monotracial C$^*$-algebra and $B$ a monotracial 
C$^*$-algebra with strict comparison, and let $\Phi$ be a homomorphism from 
$A$ to $M(B)^{\omega}$ such that $\tau_{A}=\tau_{B, \omega}\circ \Phi$. 
Suppose that $e$ and $f$ are positive contractions in $B^{\omega}\cap \Phi(A)^{\prime}$ satisfying   
$$
\tau_{B, \omega}(e)=0 \quad \text{and} \quad \inf_{m\in\mathbb{N}}\tau_{B, \omega}(f^m)>0.
$$
If $b$ is a positive element in $A$ of norm one, then there exists an element $r$ in $B^{\omega}$ 
such that
$$
fr=r, \quad \Phi(b)r=r \quad \text{and} \quad r^*r=e.
$$
\end{lem}
\begin{proof}
By similar arguments as in the proof of \cite[Lemma 3.6]{Sza6} and \cite[Section 2]{MS}, 
we can prove this lemma. We shall give a sketch of a proof for the reader's convenience. 

Let $\varepsilon>0$. It is enough to show that there exists a contraction $r$ in $B^{\omega}$ 
such that 
$$
fr=r, \quad \| \Phi(b)r-r\|< \varepsilon \quad \text{and} \quad r^*r=e
$$
by the diagonal argument. 
Take a natural number $k$ such that 
$
\| b^{k+1}- b^k\| < \varepsilon.
$
Essentially the same argument as in the proof of \cite[Lemma 2.3]{MS} shows 
that  there exists a positive contraction $\tilde{f}$ in $B^{\omega}\cap \Phi(A)^{\prime}$ 
such that $\tilde{f}f=\tilde{f}$ and 
$\inf_{m\in\mathbb{N}}\tau_{B, \omega}(f^m)=\inf_{m\in\mathbb{N}}\tau_{B, \omega}(\tilde{f}^m)$. 
Note that $\tilde{f}^{1/2}\Phi (b^{2k})\tilde{f}^{1/2}$ is a positive contraction in $B^{\omega}$. 
Since $B$ has strict comparison and $b^{2k}$ is a positive element of norm one, 
using \cite[Lemma 5.7]{Na1} instead of \cite[Lemma 2.4]{MS}, 
essentially the same argument as in the proof of \cite[Lemma 2.5]{MS} shows that 
there exists a contraction $t$ in $B^{\omega}$ such that 
$$
t^* \tilde{f}^{1/2}\Phi (b^{2k})\tilde{f}^{1/2}t = e.
$$
Put $r:= \Phi(b^{k})\tilde{f}^{1/2}t\in B^{\omega}$, then we have 
$$
fr=r, \quad \| \Phi(b)r-r\|< \varepsilon \quad \text{and} \quad r^*r=e.
$$
\end{proof}

Let $\Phi$ be a homomorphism from $A$ to $M(B)^{\omega}$ such that 
$\tau_{A}=\tau_{B, \omega}\circ \Phi$. 
If $A$ is non-unital, then $\Phi$ can be uniquely extended to a unital 
homomorphism $\Phi^{\sim}$ from $A^{\sim}$ to $M(B)^{\omega}$. 
As a notation,  put $\Phi^{\sim}=\Phi$ if $A$ is unital. 
The following lemma is essentially based on \cite[Lemma 3.6]{Sza6}. 
We refer the reader to \cite[Lemma 3.2]{MS3} and \cite[Proposition 2.2]{MS} 
for the pioneering works of this lemma. 

\begin{lem}\label{lem:szabo}
Let $A$ be a simple separable non-type I nuclear monotracial C$^*$-algebra and $B$ a monotracial 
C$^*$-algebra with strict comparison, and let $\Phi$ be a homomorphism from 
$A$ to $M(B)^{\omega}$ such that $\tau_{A}=\tau_{B, \omega}\circ \Phi$. 
Assume that $\Phi$ is unital if $A$ is unital. 
Suppose that $e$ and $f$ are positive contractions 
in $B^{\omega}\cap \Phi(A)^{\prime}$ satisfying   
$$
\tau_{B, \omega}(e)=0 \quad \text{and} \quad \inf_{m\in\mathbb{N}}\tau_{B, \omega}(f^m)>0.
$$
Then for any finite subset $F\subset A^{\sim}$ and $\varepsilon>0$,  there exist a finite subset 
$G\subset A^{\sim}$ and $\delta>0$ such that the following holds. 
If $b$ is a positive element in $A$ of norm one satisfying 
$$
\| [b, a] \| < \delta \quad \text{and} \quad  \| ba\| > \| a\| -\delta 
$$
for any $a\in G$, then there exists an element $s$ in $B^{\omega}$ such that 
$$
fs=s, \quad \|\Phi (b)s - s\|< \varepsilon \quad \text{and} \quad 
\| s^*\Phi^{\sim}(a)s-\Phi^{\sim} (a)e\| < \varepsilon
$$ 
for any $a\in F$. 
\end{lem}
\begin{proof}
By similar arguments as in the proof of \cite[Lemma 3.6]{Sza6}, \cite[Lemma 3.2]{MS3} and 
\cite[Proposition 2.2]{MS}, we can prove this lemma. We shall give a sketch of a proof for the 
reader's convenience. 

Let $F\subset  A^{\sim}$ be a finite subset and $\varepsilon>0$. We may assume that 
every element in $F$ is of norm one. 
Take a pure state $\lambda$ on $A$. Then $\pi_{\lambda}(A)\cap K(H_{\lambda})=\{0\}$
because $A$ is simple, separable and of non-type I. 
Note that $\lambda$ can be uniquely extended to a pure state 
$\lambda^{\sim}$ on $A^{\sim}$ and we have $\pi_{\lambda^{\sim}}(A^{\sim})\cap K(H_{\lambda})=\{0\}$. 
Hence $\mathrm{id}_{A^{\sim}}$ is a point-norm limit of a sequence of one-step-elementary maps 
by nuclearity of $A$ and \cite[Proposition 5.9]{KR}. 
In particular, the proof of \cite[Proposition 5.9]{KR} shows that there exist elements 
$c_1,...,c_N$, $d_1,...,d_N$ in $A^{\sim}$ such that 
$$
\| a- \sum_{i,j=1}^{N}\lambda^{\sim} (d_i^*ad_j)c_i^*c_j \| < \dfrac{\varepsilon}{3}
$$
for any $a\in F$. We may assume that $c_i$ is a contraction for any $1\leq i \leq N$ as in the proof 
of \cite[Proposition 2.2]{MS}. 
By \cite[Proposition 2.2]{AAP}, there exists a positive contraction $a_0$ in $A^{\sim}$ such that 
$\lambda^{\sim} (a_0)=1$ and 
$$
\| a_0d_i^*ad_j a_0 - \lambda^{\sim}(d_i^*ad_j)a_0^2\| < \dfrac{\varepsilon}{3N^2}
$$
for any $a\in F$ and $1\leq i,j\leq N$. 
Put 
$$
G:= \{ d_ia_0 \; |\; 1\leq i \leq N\}\cup \{a_0\}
$$
and
$$
\delta:=\min\left\{\dfrac{1}{2},\; 
\dfrac{\varepsilon^2}{(N+ \sum_{i=1}^{N} \sqrt{2} \| d_i \|)^2},\; \dfrac{\varepsilon}{6}\right\}.
$$
We shall show that $G$ and $\delta$ satisfy the desired property. Let $b$ be a positive element  
in $A$ of norm one such that 
$\| [b, a] \| < \delta$ and $\| ba\| > \| a\| -\delta$ for any $a\in G$. 
Note that we have $0<(1-\delta)^2<\| ba_0^2b \| \leq 1$ since $\| ba_0\| > \| a_0\| -\delta = 1-\delta>0$. 
Lemma \ref{lem:property-si} implies that there exists a contraction $r$ in $B^{\omega}$ such that 
$$
fr=r, \quad \Phi\left(\dfrac{ba_0^2b}{\|b a_0^2 b\|}\right)r=r \quad \text{and} \quad r^*r=e. 
$$
Put 
$$
s:= \sum_{i=1}^{N}\Phi (d_i a_0 b)r\Phi^{\sim} (c_i)\in B^{\omega}, 
$$
then $fs=s$. Since we have 
\begin{align*}
\| \Phi(b)r-\Phi(b^2)r\|^2
&\leq \|r-\Phi(b)r \|^2 
= \| r^*\Phi^{\sim}(1_{A^{\sim}}-b)^2r\| \leq \|r^* \Phi^{\sim} (1_{A^{\sim}}-b)r\|  \\
& \leq  \|r^* \Phi^{\sim} (1_{A^{\sim}}-b^2)r\|
\leq 
\|r^* \Phi^{\sim}(1_{A^{\sim}}-ba_0^2b) r\| 
\leq \| r-\| ba_0^2 b \|r \| \\
&\leq 1-\| ba_0^2b\|  
 <1 - (1-\delta)^2 < 2\delta,
\end{align*}
\begin{align*}
\| \Phi(b) s - s\| 
& = \| \sum_{i=1}^{N}\Phi (bd_i a_0 b)r\Phi^{\sim} (c_i) - \sum_{i=1}^{N}\Phi (d_i a_0 b^2)r\Phi^{\sim} (c_i) \\
& \quad +\sum_{i=1}^{N}\Phi (d_i a_0 b^2)r\Phi^{\sim} (c_i) 
-\sum_{i=1}^{N}\Phi (d_i a_0 b)r\Phi^{\sim} (c_i)\| \\
& \leq \sum_{i=1}^{N} \| [b, d_ia_0] \| + \sum_{i=1}^{N} \| d_i\| \| \Phi (b^2)r -\Phi (b)r \| \\
& < N\delta + \sum_{i=1}^{N} \| d_i \| \sqrt{2\delta}< (N+ \sum_{i=1}^{N} \sqrt{2} \| d_i \| ) 
\sqrt{\delta} \leq \varepsilon.
\end{align*}
For any $a\in F$,  we have 
\begin{align*}
& \| s^*\Phi^{\sim}(a)s-\Phi^{\sim} (a)e\| \\
& = \|   \sum_{i,j=1}^{N}\Phi^{\sim} (c_i^*)r^*\Phi (ba_0d_i^*ad_j a_0 b)r\Phi^{\sim} (c_j) 
-\Phi^{\sim} (a)e  \| \\
& < \| \sum_{i,j=1}^{N}\Phi^{\sim} (c_i^*)r^*\Phi (b\lambda^{\sim}(d_i^*ad_j) a_0^2 b)r\Phi^{\sim} (c_j) 
-\Phi^{\sim} (a)e  \| + \dfrac{\varepsilon}{3} \\
& = \| \sum_{i,j=1}^{N}\lambda^{\sim}(d_i^*ad_j)\| ba_0^2b\| 
\Phi^{\sim} (c_i^*)r^*r\Phi^{\sim} (c_j)- \Phi^{\sim}(a)e\| +  \dfrac{\varepsilon}{3} \\
& = \| \|ba_0^2b\|  \Phi^{\sim} \left(\sum_{i,j=1}^{N}\lambda^{\sim}(d_i^*ad_j)c_i^*c_j\right)e - 
\Phi^{\sim}(a)e\| +  \dfrac{\varepsilon}{3} \\
& = \|\|ba_0^2b\|  \Phi^{\sim} \left(\sum_{i,j=1}^{N}\lambda^{\sim}(d_i^*ad_j)c_i^*c_j-a\right)e 
- (1-\|ba_0^2b\|)\Phi^{\sim}(a)e\| 
+\dfrac{\varepsilon}{3} \\
& < \dfrac{\varepsilon}{3} + 1-(1-\delta)^2+\dfrac{\varepsilon}{3} < 
\dfrac{2}{3}\varepsilon +2\delta \leq \varepsilon. 
\end{align*}
Therefore the proof is complete. 
\end{proof}

The following proposition is an immediate consequence of Sato's observation 
\cite[Proposition 2.1]{Sa2}. See also \cite[Theorem 2.1]{K}, \cite[Proposition 4.1]{Sza6} and the proof of 
\cite[Proposition 5.1]{Sa2}.  

\begin{pro}\label{pro:sato}
Let $A$ be a simple separable C$^*$-algebra, and let $\{\alpha_{i}\}_{i\in\mathbb{N}}$ be a sequence 
of outer automorphisms of $A$. Then there exists a sequence $\{b_n\}_{n\in\mathbb{N}}$ 
of positive elements in $A$ of norm one such that 
$$
\lim_{n\to \infty} \| [b_n, a] \|=0, \quad \lim_{n\to \infty}\| b_na \|= \| a\| \quad \text{and} \quad 
\lim_{n\to \infty} \| b_n \alpha_i (b_n)\|=0
$$
for any $a\in A^{\sim}$ and $i\in\mathbb{N}$. 
\end{pro}

\begin{rem}
In the proposition above, we consider that $a$ is an element in the the unitization algebra $A^{\sim}$. 
Note that for a state $\lambda$ on a simple C$^*$-algebra $A$, we have 
$$
\| c_{\lambda}a \| = \| a\| 
$$
for any $a\in A^{\sim}$ where $c_{\lambda}$ is the central cover of $\pi_{\lambda}$ in $A^{**}$. 
Indeed, define a homomorphism $\pi$ from $A^{\sim}$ to $A^{**}$ by 
$\pi(a)=c_{\lambda}a$ for any $a\in A^{\sim}$. Then it is easy to see that 
$\mathrm{ker}\; \pi= \{0\}$, and hence $\pi$ is isometric. 
\end{rem}

The following theorems are essentially based on \cite[Proposition 5.1]{Sa2} and \cite[Theorem 4.2]{Sza6}. 

\begin{thm}\label{thm:si}
Let $A$ be a simple separable non-type I nuclear monotracial C$^*$-algebra and $B$ a monotracial 
C$^*$-algebra with strict comparison, and let $\Phi$ be a homomorphism from 
$A$ to $M(B)^{\omega}$ such that $\tau_{A}=\tau_{B, \omega}\circ \Phi$. 
Assume that $\Phi$ is unital if $A$ is unital. 
Let $\beta^{\omega}$ be a semiliftable action of a countable discrete amenable group on 
$M(B)^{\omega}$ such that $\beta^{\omega}_g(\Phi(A))=\Phi(A)$ for any $g\in \Gamma$. 
Suppose that $\beta^{\omega}|_{\Phi(A)}$ is outer. 
If $e$ and $f$ are positive contractions 
in $(B^{\omega}\cap \Phi(A)^{\prime})^{\beta^{\omega}}$ satisfying   
$$
\tau_{B, \omega}(e)=0 \quad \text{and} \quad \inf_{m\in\mathbb{N}}\tau_{B, \omega}(f^m)>0,
$$
then there exists an element $s$ in $(B^{\omega}\cap \Phi(A)^{\prime})^{\beta^{\omega}}$ such that 
$fs=s$ and $s^*s=e$. 
\end{thm}
\begin{proof}
By similar arguments as in the proof of \cite[Lemma 3.6]{Sza6}, \cite[Proposition 4.5]{MS2} and 
\cite[Proposition 2.2]{Sa2}, we can prove this lemma. We shall give a sketch of a proof for the 
reader's convenience. 

Using Lemma \ref{lem:szabo} and Proposition \ref{pro:sato}, we obtain sequences 
$\{r_n\}_{n\in\mathbb{N}}$ of elements in $B^{\omega}$ and $\{b_n\}_{n\in\mathbb{N}}$ of 
positive elements in $A$ of norm one such that  
$$
fr_n= r_n, \quad \lim_{n\to \infty}\| \Phi (b_n) r_n -r_n \|=0, \quad  
\lim_{n\to\infty} \| r_n^* \Phi^{\sim}(a) r_n - \Phi^{\sim}(a)e\| =0
$$
and 
$$
\lim_{n\to \infty} \| \Phi(b_n) \beta^{\omega}_g (\Phi(b_n))\|=0
$$
for any $a\in A^{\sim}$ and $g\in\Gamma\setminus\{\iota\}$. 
Note that we have $\lim_{n\to \infty}\| r_n^*r_n - e\|=0$ since $\Phi^{\sim}$ is unital. 
Also, we have 
$
\lim_{n\to \infty} \| r_n^* \beta^{\omega}_g (r_n) \| =0
$
for any $g\in\Gamma\setminus\{\iota\}$. By the diagonal argument, there exists a contraction 
$r$ in $B^{\omega}$ such that 
$$
fr=r, \quad r^*\Phi (a) r=\Phi (a)e, \quad r^*r=e \quad \text{and} \quad r^*\beta^{\omega}_g(r)=0
$$ 
for any $a\in A$ and $g\in\Gamma\setminus \{\iota\}$.  Since we have 
\begin{align*}
[r, \Phi (a)]^*[r, \Phi(a)]
&=\Phi(a^*)r^*r\Phi(a)-\Phi(a^*)r^*\Phi(a)r-r^*\Phi(a^*)r\Phi(a)+r^*\Phi(a^*a)r \\
&= \Phi(a^*)e\Phi(a)-\Phi(a^*)\Phi(a)e-\Phi(a^*)e\Phi(a)+\Phi(a^*a)e=0
\end{align*}
for any $a\in A$, $r$ is an element in $B^{\omega}\cap \Phi(A)^{\prime}$. 

Let $\varepsilon>0$ and a finite subset $\Gamma_0 \subset \Gamma$. 
It is enough to show that there exists an element $s$ in 
$B^{\omega}\cap \Phi(A)^{\prime}$ 
such that 
$$
fs=s, \quad s^*s=e  \quad \text{and} \quad \| s - \beta^{\omega}_h(s)\| < \varepsilon
$$
for any $h\in \Gamma_0$ by the diagonal argument. Since $\Gamma$ is amenable, 
there exists a finite subset $\Gamma_1\subset \Gamma$ such that 
$$
\max_{h\in \Gamma_0} \dfrac{|(h\Gamma_1\setminus \Gamma_1)\cup (\Gamma_1\setminus 
h\Gamma_1)|}{|\Gamma_1|}< \varepsilon^2
$$
where $|(h\Gamma_1\setminus \Gamma_1)\cup (\Gamma_1\setminus 
h\Gamma_1)|$ and $|\Gamma_1|$ are the cardinality of 
$(h\Gamma_1\setminus \Gamma_1)\cup (\Gamma_1\setminus h\Gamma_1)$ and 
$\Gamma_1$, respectively. 
Put
$$
s:= \dfrac{1}{\sqrt{|\Gamma_1|}}\sum_{g\in \Gamma_1}\beta^{\omega}_g (r)\in 
B^{\omega}\cap \Phi(A)^{\prime}.
$$
Then $fs=s$ and $s^*s=e$. Also, we have 
\begin{align*}
\| s- \beta^{\omega}_h(s)\| 
&=\|(s-\beta^{\omega}_h(s))^*(s-\beta^{\omega}_h(s)) \|^{1/2} \\
&= \left\| \dfrac{1}{|\Gamma_1|}\left(
\sum_{g\in h\Gamma_1\setminus \Gamma_1}
\beta^{\omega}_g(r^*r) +\sum_{g\in \Gamma_1\setminus h\Gamma_1}\beta^{\omega}_g(r^*r)\right)
\right\|^{1/2} \\
&=
\left\|\dfrac{|(h\Gamma_1\setminus \Gamma_1)\cup (\Gamma_1\setminus h\Gamma_1)|}
{|\Gamma_1|}e\right\|^{1/2}< \varepsilon
\end{align*}
for any $h\in\Gamma_0$. Therefore we obtain the conclusion. 
\end{proof}

Essentially the same argument as above (or as in the proof of \cite[Theorem 4.2]{Sza6}) 
shows the following theorem. 

\begin{thm}\label{thm:kir-si}
Let $A$ be a simple separable non-type I nuclear monotracial C$^*$-algebra and $B$ a monotracial 
C$^*$-algebra with strict comparison, and let $\Phi$ be a homomorphism from 
$A$ to $M(B)^{\omega}$ such that $\tau_{A}=\tau_{B, \omega}\circ \Phi$. 
Assume that $\Phi$ is unital if $A$ is unital. 
Let $\beta^{\omega}$ be a semiliftable action of a countable discrete amenable group on 
$F(\Phi (A), B)$.
Suppose that $\beta^{\omega}|_{\Phi(A)}$ is outer. 
Then $\Phi$ has property (SI) relative to $\beta^{\omega}$, that is, 
if $e$ and $f$ are positive contractions in $F(\Phi (A), B)^{\beta^{\omega}}$ satisfying   
$$
\tau_{B, \omega}(e)=0 \quad \text{and} \quad \inf_{m\in\mathbb{N}}\tau_{B, \omega}(f^m)>0,
$$
then there exists an element $s$ in $F(\Phi (A), B)^{\beta^{\omega}}$ such that 
$fs=s$ and $s^*s=e$. 
\end{thm}

Let $A$ be a simple separable nuclear monotracial C$^*$-algebra and $B$ a monotracial 
C$^*$-algebra, and let $\Phi$ be a homomorphism from 
$A$ to $M(B)^{\omega}$ such that $\tau_{A}=\tau_{B, \omega}\circ \Phi$. Assume that $\tau_{B}$ 
is faithful.
Put
$$
M:= \pi_{\tau_{B}}(B)^{''}
$$
and 
$$
\mathcal{M}:= \ell^{\infty}(\mathbb{N}, M)/\{\{x_n\}_{n\in\mathbb{N} }
\in \ell^{\infty}(\mathbb{N}, M)\; |\; \lim_{n\to\omega}\tilde{\tau}_{B}(x_n^*x_n)=0\}
$$
where $\tilde{\tau}_{B}$ is the unique normal extension of $\tau_{B}$ on $M$. 
Note that $\mathcal{M}$ is a von Neumann algebraic ultrapower of $M$, and 
hence $\mathcal{M}$ is a finite factor. 
We identify $M$ with the subalgebra of $\mathcal{M}$ consisting of equivalence 
classes of constant sequences, and set 
$$
M_{\omega}:= \mathcal{M}\cap M^{\prime}. 
$$
Define a homomorphism  $\varrho$ from $B^{\omega}$ to $\mathcal{M}$ by $\varrho ((x_n)_n) =
(\pi_{\tau_B}(x_n)_n)_n$. By Kaplansky's density theorem, we see that $\varrho$ is surjective. 
Since the GNS representation $\pi_{\tau_B}$ of $B$ on $H_{\tau_{B}}$ can be extended to a 
representation of $M(B)$ on $H_{\tau_{B}}$, $\varrho$ can be extended to a homomorphism 
from $M(B)^{\omega}$ to $\mathcal{M}$. 
We denote it by the same symbol $\varrho$ for simplicity. 
Set 
$$
\mathcal{M}(\Phi (A), B):= \mathcal{M}\cap \varrho (\Phi(A))^{\prime}. 
$$
Then $\varrho$ maps $B^{\omega}\cap \Phi (A)^{\prime}$ into $\mathcal{M}(\Phi (A), B)$. 
Essentially the same proof as \cite[Theorem 3.1]{MS3} (see also \cite[Theorem 3.3]{KR} and 
\cite[Proposition 3.4]{Na3}) shows the following proposition. 

\begin{pro}\label{pro:surjective}
With notation as above, 
$\varrho|_{B^{\omega}\cap \Phi (A)^{\prime}}: B^{\omega}\cap 
\Phi (A)^{\prime}\to \mathcal{M}(\Phi (A), B)$ is surjective. 
\end{pro}

The same proof as \cite[Proposition 2.1]{Na4} shows 
$\mathrm{Ann}(\Phi (A), B^{\omega})\subseteq \mathrm{ker}\; \varrho$. Hence 
$\varrho$ induces a homomorphism $[\varrho]$ from $F(\Phi(A), B)$ to $\mathcal{M}(\Phi (A), B)$. 
The proposition above implies that $[\varrho]$ is surjective. 
Essentially the same proof as \cite[Proposition 2.5]{Na4} shows the following proposition. 

\begin{pro}\label{pro:factor}
Let $A$ be a simple separable nuclear monotracial C$^*$-algebra and $B$ a monotracial 
C$^*$-algebra, and let $\Phi$ be a homomorphism from 
$A$ to $M(B)^{\omega}$ such that $\tau_{A}=\tau_{B, \omega}\circ \Phi$. Assume that $\tau_{B}$ 
is faithful. Then $\mathcal{M}(\Phi (A), B)$ is a finite factor. 
\end{pro}

Let $\beta^{\omega}$ be a semiliftable action of $\Gamma$ on $M(B)^{\omega}$ such that  
$\beta_g^{\omega}(\Phi(A))=\Phi (A)$ for any $g\in \Gamma$. 
Since $B$ is monotracial, 
we have $\beta_{g}^{\omega}(\mathrm{ker}\; \varrho)=\mathrm{ker}\; \varrho$ 
for any $g\in\Gamma$. 
Hence $\beta^{\omega}$ induces actions on $\mathcal{M}$ and 
$\mathcal{M}(\Phi (A), B)$. We denote them by $\tilde{\beta}^{\omega}$. 
Note that we have $\varrho \circ \beta_g^{\omega}= \tilde{\beta}^{\omega}_g \circ \varrho$ for any 
$g\in\Gamma$. 

\begin{pro}\label{pro:equiv-surjective}
With notation as above,  $\varrho|_{(B^{\omega}\cap \Phi (A)^{\prime})^{\beta^{\omega}}}: (B^{\omega}\cap 
\Phi (A)^{\prime})^{\beta^{\omega}} \to \mathcal{M}(\Phi (A), B)^{\tilde{\beta}^{\omega}}$ and 
$[\varrho]|_{F(\Phi (A), B)^{\beta^{\omega}}}: F( \Phi (A), B)^{\beta^{\omega}} \to 
\mathcal{M}(\Phi (A), B)^{\tilde{\beta}^{\omega}}$ are surjective. 
\end{pro}
\begin{proof}
Let $y\in \mathcal{M}(\Phi (A), B)^{\tilde{\beta}^{\omega}}$, then 
there exists an element $x$ in $B^{\omega}\cap \Phi (A)^{\prime}$ such that $\varrho (x)=y$ 
by Proposition \ref{pro:surjective}. 
Define $C$ to be a C$^*$-subalgebra of $B^{\omega}\cap \Phi (A)^{\prime}$
generated by $\{\beta_g^{\omega} (x)\; |\; g\in \Gamma \}$, and put 
$J:=\mathrm{ker}\; \varrho|_{B^{\omega}\cap \Phi (A)^{\prime}}$. 
Since $J\cap C$ is a C$^*$-subalgebra of a separable C$^*$-algebra $C$ and $\beta^{\omega}$ 
induces an action on $C$, \cite[Lemma 1.4]{Kas} ($\varphi$ is taken to be the zero map) implies that 
there exists a sequence $\{e_n\}_{n\in\mathbb{N}}$ of 
positive contractions in $J\cap C$ such that 
$$
\lim_{n\to\infty}\|e_na-a \|=0 \quad \text{and} \quad 
\lim_{n\to\infty} \| \beta_g^{\omega} (e_n)-e_n \| =0
$$
for any $a\in J\cap C$ and $g\in \Gamma$. 
By the diagonal argument, we obtain a positive contraction $e$ in 
$J^{\beta^{\omega}}$ such that $ea=a$ for any $a\in J\cap C$. 
Put $z:= x-ex$, then we have $z\in (B^{\omega}\cap \Phi (A)^{\prime})^{\beta^{\omega}}$ and 
$\varrho (z)=y$. Therefore $\varrho|_{(B^{\omega}\cap \Phi (A)^{\prime})^{\beta^{\omega}}}: 
(B^{\omega}\cap \Phi (A)^{\prime})^{\beta^{\omega}} \to 
\mathcal{M}(\Phi (A), B)^{\tilde{\beta}^{\omega}}$ is surjective. 
Also, this implies that $[\varrho]|_{F(\Phi (A), B)^{\beta^{\omega}}}: F( \Phi (A), B)^{\beta^{\omega}} \to 
\mathcal{M}(\Phi (A), B)^{\tilde{\beta}^{\omega}}$ is surjective.
\end{proof}

Using Theorem \ref{thm:si} or Theorem \ref{thm:kir-si} and Proposition \ref{pro:equiv-surjective}
instead of \cite[Proposition 4.5]{MS2} and \cite[Theorem 4.3]{MS2}, 
the same argument as the proofs of \cite[Lemma 4.7]{MS2} and \cite[Proposition 4.8]{MS2} 
show the following propositions. 
See also the proofs of \cite[Proposition 3.3]{MS3} and \cite[Proposition 3.8]{Na3}.

\begin{pro}\label{pro:strict-comparison-notk}
Let $A$ be a simple separable non-type I nuclear monotracial C$^*$-algebra and $B$ a monotracial 
C$^*$-algebra with strict comparison, and let $\Phi$ be a homomorphism from 
$A$ to $M(B)^{\omega}$ such that $\tau_{A}=\tau_{B, \omega}\circ \Phi$. 
Assume that $\tau_{B}$ is faithful, and $\Phi$ is unital if $A$ is unital. 
Let $\beta^{\omega}$ be a semiliftable action of a countable discrete amenable group on 
$M(B)^{\omega}$ such that $\beta^{\omega}_g(\Phi(A))=\Phi(A)$ for any $g\in \Gamma$. 
Suppose that  $\beta_g^{\omega}|_{\Phi(A)}$ is outer for any $g\in \Gamma\setminus \{\iota\}$. 
If $\mathcal{M}(\Phi (A), B)^{\tilde{\beta}^{\omega}}$ is a factor, 
then $(B^{\omega}\cap \Phi(A)^{\prime})^{\beta^{\omega}}$ is 
monotracial and has strict comparison. 
Furthermore, if $a$ and $b$ are positive elements in 
$(B^{\omega}\cap \Phi(A)^{\prime})^{\beta^{\omega}}$ satisfying $d_{\tau_{B, \omega}}(a)< 
d_{\tau_{B, \omega}}(b)$, then there exists an element $r$ in 
$(B^{\omega}\cap \Phi(A)^{\prime})^{\beta^{\omega}}$ such that $r^*br=a$. 
\end{pro}

\begin{pro}\label{pro:strict-comparison}
Let $A$ be a simple separable non-type I nuclear monotracial C$^*$-algebra and $B$ a monotracial 
C$^*$-algebra with strict comparison, and let $\Phi$ be a homomorphism from 
$A$ to $M(B)^{\omega}$ such that $\tau_{A}=\tau_{B, \omega}\circ \Phi$. 
Assume that $\tau_{B}$ is faithful, and $\Phi$ is unital if $A$ is unital. 
Let $\beta^{\omega}$ be a semiliftable action of a countable discrete amenable group on 
$F(\Phi (A), B)$.  
Suppose that $\beta_g^{\omega}|_{\Phi(A)}$ is outer for any 
$g\in \Gamma\setminus \{\iota\}$. If $\mathcal{M}(\Phi (A), B)^{\tilde{\beta}^{\omega}}$ is a 
factor, then $F(\Phi (A), B)^{\beta^{\omega}}$ is monotracial and has strict comparison. 
Furthermore, if $a$ and $b$ are positive elements in 
$F(\Phi (A), B)^{\beta^{\omega}}$ satisfying $d_{\tau_{B, \omega}}(a)< 
d_{\tau_{B, \omega}}(b)$, then there exists an element $r$ in 
$F(\Phi (A), B)^{\beta^{\omega}}$ such that $r^*br=a$. 
\end{pro}

The following proposition is an immediate corollary of \cite[Lemma 4.1]{MS2}. 
Note that we do not assume that $\beta$ is strongly outer. 
We shall give a proof for the reader's convenience. 

\begin{pro}\label{pro:MS-lemma-4.1}
Let $B$ be a simple separable non-type I nuclear monotracial C$^*$-algebra, and let $\beta$ be 
an action of a countable discrete amenable group $\Gamma$ on $B$. 
Then $M_{\omega}^{\tilde{\beta}}=(\pi_{\tau_{B}}(B)^{''}_{\omega})^{\tilde{\beta}}$ is a II$_1$ factor. 
\end{pro}
\begin{proof}
Since $B$ is a simple separable non-type I nuclear monotracial C$^*$-algebra, 
$M=\pi_{\tau_{B}}(B)^{''}$ is isomorphic to the injective II$_1$ factor. 
Put 
$$
\Gamma_0=\{g\in\Gamma\; |\; \tilde{\beta}_g\text{ is an inner automorphism of } M \}.
$$ 
Then 
$\Gamma_0$ is a normal subgroup of $\Gamma$. Since we have 
$\tilde{\beta}_g=\mathrm{id}_{M_{\omega}}$ in $M_{\omega}$ for any $g\in \Gamma_0$, $\tilde{\beta}$ induces 
an action $\gamma$ of $\Gamma/\Gamma_0$ on $M_{\omega}$ such that $M_{\omega}^{\tilde{\beta}}
=M_{\omega}^{\gamma}$. Since $\Gamma/\Gamma_0$ is amenable, 
\cite[Lemma 4.1]{MS2} (see also \cite[Proposition 7.2]{Oc}) implies that 
$M_{\omega}^{\gamma}$ is a II$_1$ factor. Consequently, we obtain the conclusion. 
\end{proof}

In the rest of this section,  we shall consider certain semiliftable actions. 
Suppose that $A$ is a simple separable monotracial C$^*$-algebra and $B$ is 
a separable nuclear monotracial C$^*$-algebra. 
Let $\gamma$ be a strongly outer action of $\Gamma$ on $B$, 
and let $\{U_{g}\; |\; g\in\Gamma \}$ be a set of unitary elements in $(B^{\sim})^{\omega}$ 
such that  a map $\beta^{\omega}:g\mapsto \mathrm{Ad}(U_g)\circ \gamma_g$ 
defines an action of $\Gamma$ on $F(\Phi (A), B)$. 
Note that we have $\tilde{\beta}^{\omega}_{g}=\mathrm{Ad}(\varrho (U_g))\circ \tilde{\gamma}_g$ 
for any $g\in\Gamma$ where $\tilde{\beta}^{\omega}$ and $\tilde{\gamma}$ are 
the induced actions on $\mathcal{M}(\Phi (A), B)$ and on $\mathcal{M}$, respectively. 
Since $B$ is separable, nuclear and monotracial, $\pi_{\tau_{B}}(B)^{''}$ is isomorphic to the 
injective II$_1$ factor. (Note that if $B$ were of type I, then there exist no 
strongly outer actions on $B$.) 
The following proposition can be regarded as a corollary of Ocneanu's first cohomology vanishing result 
\cite[Proposition 7.2]{Oc} (see also \cite[Lemma 4.9]{Masuda}). 
This result is based on Ocneanu's Rohlin theorem \cite[Theorem 6.1]{Oc} 
(see also \cite[Theorem 4.8]{Masuda}) and a Shapiro type argument. 
See also \cite[Section 3.2]{MT} for these arguments. 

\begin{pro}\label{pro:cohomology-vanishing}
Let $A$ be a simple separable monotracial C$^*$-algebra and $B$ a separable nuclear monotracial 
C$^*$-algebra, and let $\Phi$ be a homomorphism from $A$ to $M(B)^{\omega}$ 
such that $\tau_{A}=\tau_{B, \omega}\circ \Phi$. Assume that $\tau_{B}$ is faithful. 
Let $\gamma$ be a strongly outer action of  a countable discrete amenable group $\Gamma$ on 
$B$, 
and let $\{U_{g}\; |\; g\in\Gamma \}$ be a set of unitary elements in $(B^{\sim})^{\omega}$ 
such that  a map $\beta^{\omega}:g\mapsto \mathrm{Ad}(U_g)\circ \gamma_g$ 
defines an action of $\Gamma$ on $F(\Phi (A), B)$. 
Then every $\tilde{\beta}^{\omega}$-cocycle $V$ on $\mathcal{M}(\Phi (A), B)$ is a coboundary. 
\end{pro}
\begin{proof}
Since $\tilde{\gamma}$ is an outer action on the injective II$_1$ factor $M=\pi_{\tau_{B}}(B)^{''}$, 
$\tilde{\gamma}_g$ is a non-trivial automorphism of $M_{\omega}$ for any 
$g\in\Gamma\setminus\{\iota\}$ by \cite[Theorem 3.2]{C3}. 
Therefore \cite[Lemma 5.6]{Oc} implies that the action $\tilde{\gamma}$ on $M_{\omega}$ 
satisfies the assumption of Ocneanu's Rohlin theorem \cite[Theorem 6.1]{Oc} 
(see also \cite[Theorem 4.8]{Masuda} and \cite[Theorem 3.4]{MT}). 
Hence there exist ``Rohlin projections'' $\{P_g\}_{g\in G}$ in $M_{\omega}$ by 
\cite[Theorem 3.4]{MT} (which is a simplified version of \cite[Theorem 6.1]{Oc}), that is, 
$\{P_g\}_{g\in G}$ is a partition unity in $M_{\omega}$ consisting of projections such that 
$\tilde{\gamma}_g(P_h)-P_{gh}$ is small in a suitable sense. 
Taking suitable subsequences (or the fast reindexation trick \cite[Lemma 5.3]{Oc}), we may assume that 
$\{P_g\}_{g\in G}$ is contained in 
$$
M_{\omega}\cap\{V_g\; |\; g\in\Gamma\}^{\prime}\cap
\varrho (\Phi(A)\cup\{U_g\; |\; g\in\Gamma\})^{\prime} 
=\mathcal{M}(\Phi (A), B)\cap M^{\prime}\cap 
\{\varrho (U_g), V_g\; |\; g\in\Gamma\}^{\prime}.
$$ 
Note that $\{P_g\}_{g\in G}$ are also Rohlin projections for $\tilde{\beta}^{\omega}$ because we have 
$\{P_g\}_{g\in G}\subset \{\varrho (U_g)\; |\; g\in\Gamma\}^{\prime}$. 
We obtain the conclusion by a Shapiro type argument. 
For simplicity, we explain this argument in the case 
where $\Gamma$ is a finite group and $\tilde{\beta}^{\omega}_g (P_h)=P_{gh}$ for any $g,h\in \Gamma$. 
(See the proof of \cite[Proposition 7.2]{Oc} or \cite[Lemma 4.9]{Masuda} 
for the general case. We need to use the reindexing argument and consider careful estimates.) 
Put $W:=\sum_{g\in \Gamma}V_gP_g$. 
Then $W$ is a unitary element in $\mathcal{M}(\Phi (A), B)$ 
because $V_g$ is a unitary element in $\mathcal{M}(\Phi (A), B)$ and $P_g$ is a projection in 
$\mathcal{M}(\Phi (A), B)\cap \{V_g\; |\; g\in\Gamma\}^{\prime}$ for any $g\in \Gamma$.  
Since we have 
$
V_g\tilde{\beta}^{\omega}_{g}(W)= W
$
for any $g\in \Gamma$, $V$ is a coboundary. 
\end{proof}

The following proposition is based on the proof of \cite[Lemma 4.1]{MS2}.
Also, this proposition can be regarded as a corollary of Ocneanu's results. 

\begin{pro}\label{pro:relative-factor}
Let $A$ be a simple separable monotracial C$^*$-algebra and $B$ a separable nuclear monotracial 
C$^*$-algebra, and let $\Phi$ be a homomorphism from $A$ to $M(B)^{\omega}$ 
such that $\tau_{A}=\tau_{B, \omega}\circ \Phi$.  
Assume that $\tau_{B}$ is faithful. 
Let $\gamma$ be a strongly outer action of  a countable discrete amenable group $\Gamma$ on 
$B$, 
and let $\{U_{g}\; |\; g\in\Gamma \}$ be a set of unitary elements in $(B^{\sim})^{\omega}$ 
such that  a map $\beta^{\omega}:g\mapsto \mathrm{Ad}(U_g)\circ \gamma_g$ 
defines an action of $\Gamma$ on $F(\Phi (A), B)$. 
Then $\mathcal{M}(\Phi (A), B)^{\tilde{\beta}^{\omega}}$ is a II$_1$ factor. 
\end{pro}
\begin{proof}
By Proposition \ref{pro:MS-lemma-4.1}, 
$M_{\omega}^{\tilde{\gamma}}$ is a II$_1$ factor. 
Hence there exists a unital subalgebra $N$ of $M_{\omega}^{\tilde{\gamma}}$, which is 
isomorphic to the injective II$_1$ factor. Taking suitable subsequences of countably many generators 
of $N$ (or the fast reindexation trick \cite[Lemma 5.3]{Oc}), we see that 
$\mathcal{M}(\Phi (A), B)^{\tilde{\beta}^{\omega}}$ also contains 
the injective II$_1$ factor as a unital subalgebra since 
$M_{\omega}^{\tilde{\gamma}}\cap \varrho(\Phi(A)\cup\{U_g\; |\; g\in\Gamma\})^{\prime}\subset 
\mathcal{M}(\Phi (A), B)^{\tilde{\beta}^{\omega}}$ 
is a unital inclusion. 
The rest of the proof is same as the proof of \cite[Lemma 4.1]{MS2} by using 
Proposition \ref{pro:factor} and Proposition \ref{pro:cohomology-vanishing} instead of 
\cite[Theorem XIV.4.18]{Tak} and \cite[Proposition 7.2]{Oc}. 
\end{proof}

\begin{rem}\label{rem:induced}
In Proposition \ref{pro:cohomology-vanishing} and Proposition \ref{pro:relative-factor}, 
we need not assume that $\gamma$ and $\beta^{\omega}$ are genuine actions. 
It is enough to assume that $\gamma$ and $\beta^{\omega}$ induce an action 
$\tilde{\gamma}$ on $M_{\omega}$ which satisfies the assumption in \cite[Theorem 6.1]{Oc}
and an action $\tilde{\beta^{\omega}}$ on $\mathcal{M}(\Phi (A), B)$, respectively. 
\end{rem}

\section{Properties of $F(A\otimes \mathcal{W})^{\alpha\otimes\mathrm{id}_{\mathcal{W}}}$}
\label{sec:properties}

Let $A$ be a simple separable nuclear monotracial C$^*$-algebra, and let $\alpha$ be 
an outer action of a countable discrete amenable group $\Gamma$ on $A$. 
(Note that $A$ is not of type I since there exists an outer action on $A$.) 
In this section, we shall consider properties of 
$F(A\otimes\mathcal{W})^{\alpha\otimes\mathrm{id}_{\mathcal{W}}}$. 

The following proposition is an immediate corollary of 
Proposition \ref{pro:strict-comparison} and Proposition \ref{pro:MS-lemma-4.1}. 
\begin{pro}\label{pro:strict-comparison-w}
With notation as above, $F(A\otimes\mathcal{W})^{\alpha\otimes\mathrm{id}_{\mathcal{W}}}$ is 
monotracial. Furthermore, if $a$ and $b$ are positive elements in 
$F(A\otimes\mathcal{W})^{\alpha\otimes\mathrm{id}_{\mathcal{W}}}$ satisfying 
$d_{\tau_{A\otimes\mathcal{W}}, \omega}(a)<d_{\tau_{A\otimes\mathcal{W}}, \omega}(b)$, 
then there exists an element $r$ in $F(A\otimes\mathcal{W})^{\alpha\otimes\mathrm{id}_{\mathcal{W}}}$
such that $r^*br=a$. 
\end{pro}

Let $\{h_n\}_{n\in\mathbb{N}}$ be an approximate unit for $A$, and define a homomorphism 
$\Theta$ 
from $F(\mathcal{W})$ to $F(A\otimes\mathcal{W})^{\alpha\otimes\mathrm{id}_{\mathcal{W}}}$ by 
$\Theta ([(x_n)]):= [(h_n\otimes x_n)_n]$ for any $[(x_n)_n]\in F(\mathcal{W})$. 
Then $\Theta$ is a trace preserving unital homomorphism. 
Hence \cite[Proposition 4.2]{Na2} and the proposition above imply the following proposition. 

\begin{pro}\label{pro:equivariant-w}
(i) For any natural number $N$, there exists a unital homomorphism from $M_{N}(\mathbb{C})$ to 
$F(A\otimes\mathcal{W})^{\alpha\otimes\mathrm{id}_{\mathcal{W}}}$. \ \\
(ii) For any $\theta\in [0,1]$, there exists a non-zero projection $p$ in 
$F(A\otimes\mathcal{W})^{\alpha\otimes\mathrm{id}_{\mathcal{W}}}$ such that $\tau_{A\otimes\mathcal{W}, \omega}(p)=\theta$. \ \\
(iii) Let $h$ be a positive element in $F(A\otimes\mathcal{W})^{\alpha\otimes\mathrm{id}_{\mathcal{W}}}$ 
such that $d_{\tau_{A\otimes\mathcal{W}, \omega}}(h)>0$. For any 
$\theta \in [0, d_{\tau_{\omega}}(h))$, 
there exists a non-zero projection $p$ in 
$\overline{hF(A\otimes\mathcal{W})^{\alpha\otimes\mathrm{id}_{\mathcal{W}}}h}$ such that 
$\tau_{A\otimes\mathcal{W}, \omega}(p)=\theta$. 
\end{pro}

Define a homomorphism $\Phi$ from $A$ to $M(A\otimes \mathcal{W})^{\omega}$ by 
$\Phi (a):= (a\otimes 1_{\mathcal{W}^{\sim}})_n$ for any $a\in A$. 
Then we have $\tau_{A\otimes\mathcal{W}} \circ \Phi = \tau_{A}$.

\begin{pro}\label{pro:strict-comparison-stable-uniqueness}
With notation as above, 
$F(\Phi (A), A\otimes\mathcal{W})^{\alpha\otimes\mathrm{id}_{\mathcal{W}}}$ is monotracial and 
has strict comparison. 
\end{pro}
\begin{proof}
By Proposition \ref{pro:strict-comparison}, it is enough to show that 
$\mathcal{M}(\Phi(A), A\otimes\mathcal{W})^{\widetilde{\alpha\otimes\mathrm{id}_{\mathcal{W}}}}$ 
is a factor. 
Put $\Gamma_0:=\{g\in \Gamma\; |\; \tilde{\alpha}_g\text{ is an inner automorphism of } 
\pi_{\tau_A}(A)^{''}\}$, then $\Gamma_0$ is a normal subgroup of $\Gamma$. 
It can be easily checked that we have 
$$
\mathcal{M}(\Phi(A), A\otimes\mathcal{W})=
\mathcal{M}\cap (\pi_{\tau_A}(A)^{''}\bar{\otimes} \mathbb{C})^{\prime}
$$
where $\mathcal{M}$ is a von Neumann algebraic ultrapower of 
$\pi_{\tau_{A}}(A)^{''}\bar{\otimes}\pi_{\tau_{\mathcal{W}}}(\mathcal{W})^{''}$. 
Hence we see that if $g\in \Gamma_0$, then  
$\widetilde{\alpha_g\otimes\mathrm{id}_{\mathcal{W}}}
=\mathrm{id}_{\mathcal{M}(\Phi(A), A\otimes\mathcal{W})}$ in 
$\mathcal{M}(\Phi(A), A\otimes\mathcal{W})$.
Therefore $\widetilde{\alpha\otimes\mathrm{id}_{\mathcal{W}}}$ induces an action 
$\tilde{\beta}$ of $\Gamma/\Gamma_0$ on $\mathcal{M}(\Phi(A), A\otimes\mathcal{W})$, 
and we have 
$$
\mathcal{M}(\Phi(A), A\otimes\mathcal{W})^{\widetilde{\alpha\otimes\mathrm{id}_{\mathcal{W}}}}
=
\mathcal{M}(\Phi(A), A\otimes\mathcal{W})^{\tilde{\beta}}.
$$
Since  $\Gamma/\Gamma_0$ is amenable, 
Proposition \ref{pro:relative-factor} (see also Remark \ref{rem:induced})  implies that 
$\mathcal{M}(\Phi(A), A\otimes\mathcal{W})^{\tilde{\beta}}$ is a factor. Therefore we obtain the 
conclusion. 
\end{proof}

Using the proposition above instead of \cite[Proposition 3.8]{Na3}, 
we shall show Corollary \ref{cor:main-section4} which is the analogous result of \cite[Corollary 5.5]{Na3}. 
Note that we need not assume that $\Gamma$ is a finite group in arguments in \cite[Section 4 and 
Section 5]{Na3}. In the proof of \cite[Corollary 5.5]{Na3}, we need this assumption for 
\cite[Proposition 3.8]{Na3}. 
By \cite[Section 4]{Na3}, there exists a homomorphism $\rho$ from 
$F(A\otimes\mathcal{W})^{\alpha\otimes\mathrm{id}_{\mathcal{W}}}\otimes\mathcal{W}$ to 
$F(\Phi (A), A\otimes\mathcal{W})^{\alpha\otimes\mathrm{id}_{\mathcal{W}}}$ such that the following 
holds. 
If $(z_n)_n$ is an element in $(A\otimes\mathcal{W})^{\omega}\cap \Phi(A)^{\prime}$ such that 
$[(z_n)_n]=\rho ([(x_n)_n]\otimes b)$ for some 
$[(x_n)_n]\in F(A\otimes\mathcal{W})^{\alpha\otimes\mathrm{id}_{\mathcal{W}}}$ and 
$b\in \mathcal{W}$, then 
$$
(z_n(a\otimes 1_{\mathcal{W}^{\sim}}))_n= (x_n(a\otimes b))_n
$$
for any $a\in A$. For a projection $p$ in $F(A\otimes\mathcal{W})^{\alpha\otimes\mathrm{id}_{\mathcal{W}}}$, put 
$$
F(\Phi (A), A\otimes\mathcal{W})^{\alpha\otimes\mathrm{id}_{\mathcal{W}}}_p := 
\overline{\rho (p\otimes h)F(\Phi (A), A\otimes\mathcal{W})^{\alpha\otimes\mathrm{id}_{\mathcal{W}}}\rho (p\otimes h)}
$$
where $h$ is a strictly positive element in $\mathcal{W}$. 

For finite sets $G_1$ and $G_2$, let $G_1\odot G_2:=\{a\otimes b\; |\; a\in G_1, b\in G_2\}$. 
Using Proposition \ref{pro:strict-comparison-stable-uniqueness}, essentially the same arguments 
as in \cite[Section 4]{Na3} or \cite[Section 3]{Na2} show the following theorem. 
(See Section \ref{sec:sza} for  the precise definition of an approximately multiplicative map.) 
This theorem is a variant of  Elliott-Gong-Lin-Niu's stable uniqueness 
theorem \cite[Corollary 3.15]{EGLN}. 

\begin{thm} (cf. \cite[Corollary 4.5]{Na3}). \ \\
Let $\Omega$ be a compact metrizable space. For any finite subsets 
$F_1\subset C(\Omega)$, $F_2\subset \mathcal{W}$ and $\varepsilon>0$, 
there exist finite subsets $G_1\subset C(\Omega)$, $G_2\subset \mathcal{W}$, 
$m\in\mathbb{N}$  and $\delta >0$ such that the following holds. 
Let $p$ be a projection in $F(A\otimes\mathcal{W})^{\alpha\otimes\mathrm{id}_{\mathcal{W}}}$ such that 
$\tau_{A\otimes\mathcal{W}, \omega} (p)>0$. 
For any contractive ($G_1\odot G_2, \delta$)-multiplicative maps 
$\varphi, \psi : C(\Omega)\otimes \mathcal{W}\to F(\Phi (A), A\otimes\mathcal{W})^{\alpha\otimes\mathrm{id}_{\mathcal{W}}}_p$, 
there exist a unitary element $u$ in 
$M_{m^2+1}(F(\Phi (A), A\otimes\mathcal{W})^{\alpha\otimes\mathrm{id}_{\mathcal{W}}}_p)^{\sim}$ 
and $z_1,z_2,...,z_m\in\Omega$ such that 
\begin{align*}
\| u & (\varphi(f\otimes b) \oplus  \overbrace{\bigoplus_{k=1}^m f(z_k)\rho (p\otimes b)\oplus \cdots \oplus\bigoplus_{k=1}^m f(z_k)\rho (p\otimes b) }^m) u^* \\
& - \psi(f\otimes b)\oplus \overbrace{\bigoplus_{k=1}^m f(z_k)\rho (p\otimes b) \oplus \cdots \oplus \bigoplus_{k=1}^m f(z_k)\rho (p\otimes b)}^m\| < \varepsilon 
\end{align*}
for any $f\in F_1$ and $b\in F_2$. 
\end{thm}

Using Proposition \ref{pro:strict-comparison-w}, Proposition \ref{pro:equivariant-w} and the theorem 
above, essentially the same arguments as in \cite[Section 5]{Na3} show the following theorem. 

\begin{thm} (cf. \cite[Theorem 5.3]{Na3}). \ \\
Let $N_1$ and $N_2$ be normal elements in 
$F(A\otimes\mathcal{W})^{\alpha\otimes\mathrm{id}_{\mathcal{W}}}$ such that 
$\mathrm{Sp} (N_1)=\mathrm{Sp} (N_2)$ and 
$\tau_{A\otimes\mathcal{W}, \omega} (f(N_1)) >0$ for any $f\in C(\mathrm{Sp}(N_1))_{+}\setminus \{0\}$. 
Then there exists a unitary element $u$ in 
$F(A\otimes\mathcal{W})^{\alpha\otimes\mathrm{id}_{\mathcal{W}}}$ 
such that $uN_1u^* =N_2$ if and only if 
$
\tau_{A\otimes\mathcal{W}, \omega} (f(N_1))= \tau_{A\otimes\mathcal{W}, \omega} (f(N_2))
$ 
for any $f\in C(\mathrm{Sp}(N_1))$. 
\end{thm}

Using Proposition \ref{pro:strict-comparison-w}, Proposition \ref{pro:equivariant-w} and 
the theorem above, essentially the same proof as \cite[Corollary 5.5]{Na3} shows the following corollary. 

\begin{cor}\label{cor:main-section4}
Let $p$ and $q$ be projections in $F(A\otimes \mathcal{W})^{\alpha\otimes\mathrm{id}_{\mathcal{W}}}$ 
such that 
$0< \tau_{A\otimes\mathcal{W}, \omega} (p) \leq 1$.  
Then $p$ and $q$ are Murray-von Neumann equivalent if and only if 
$
\tau_{A\otimes\mathcal{W}, \omega} (p)= \tau_{A\otimes\mathcal{W}, \omega} (q)
$. 
\end{cor}

\section{Szab\'o's approximate cocycle intertwining argument}\label{sec:sza}

In this section we shall consider a slight variant of Szab\'o's approximate cocycle intertwining argument in 
\cite{Sza7} (see also \cite{Ell2}). 
Throughout this section, we assume that $A$ and $B$ are non-unital C$^*$-algebras. 

Let $\alpha$ and $\beta$ be actions of  $\Gamma$ on $A$ and $B$, 
respectively, and let $\Gamma_0\subset \Gamma$, $F\subset A$, $G\subset B$ and $\varepsilon>0$. 
We say that $\Gamma_0$ is \textit{inverse-closed} if $\Gamma_0^{-1}=\Gamma_0$ and $\iota \in 
\Gamma_0$. 
A contractive completely positive (c.c.p.) 
map $\varphi :A\to B$ is said to be \textit{$(F, \varepsilon)$-multiplicative} if 
$$
\|\varphi (ab)-\varphi(a)\varphi(b)\|<\varepsilon
$$ 
for any $a, b\in F$. 
Since we assume that $A$ and $B$ are non-unital, every c.c.p. map $\varphi$ from $A$ to $B$ can be 
uniquely extended to a unital c.p. map $\varphi^{\sim}$ from $A^{\sim}$ to $B^{\sim}$ (see, for example, 
\cite[Proposition 2.2.1]{BO}). 
Hence if $\varphi$ is an $(F, \varepsilon)$-multiplicative map from $A$ to $B$, 
then $\varphi^{\sim}$ is a unital $(F^{\sim}, \varepsilon)$-multiplicative map 
from $A^{\sim}$ to $B^{\sim}$ where $F^{\sim}=\{a+\lambda 1_{A^{\sim}}\; |\; a \in F, 
\lambda\in\mathbb{C}\}$. 
We say that a map $u$ from $\Gamma$ to the unitary group of $B^{\sim}$ is a 
\textit{$(\Gamma_0, G, \varepsilon)$-approximate $\beta$-cocycle in  
the unitization algebra} if $u_\iota =1_{B^{\sim}}$ and 
$$
\| b(u_{gh}-u_{g}\beta_g(u_{h})) \|< \varepsilon
$$
for any $g,h\in \Gamma_0$ and $b\in G$.  
A \textit{proper $(\Gamma_0, F, \varepsilon)$-approximate cocycle morphism 
from $(A, \alpha)$ to $(B, \beta)$} is a pair $(\varphi, u)$, where $\varphi$ is 
an $(F, \varepsilon)$-multiplicative map from $A$ to $B$ and $u$ is 
a $(\Gamma_0, \varphi(F), \varepsilon)$-approximate $\beta$-cocycle in the unitization algebra such that 
$$
\| \varphi \circ \alpha_g (a) - \mathrm{Ad}(u_g)\circ \beta_g \circ \varphi (a) \| < \varepsilon
$$
for any $g\in \Gamma_0$ and $a\in F$.

\begin{rem}
In \cite{Sza7}, Szab\'o defined cocycle morphisms and proper cocycle morphisms for 
(twisted) group actions on C$^*$-algebras. Furthermore, he provided a categorical 
framework for the classification of group actions on C$^*$-algebras up to cocycle conjugacy. 
We do not consider a categorical framework. Our approach seems to be an ad hoc way for 
non-unital C$^*$-algebras $A$ with $A\subset \overline{\mathrm{GL}(A^{\sim})}$. 
\end{rem}

The following proposition is related to \cite[Proposition 3.4]{Sza7} and \cite[Proposition 2.3.2]{Ror1}. 

\begin{pro}\label{pro:intertwining}
Let $A$ and $B$ be separable non-unital C$^*$-algebras, and let $\alpha$ and $\beta$ be 
actions of a countable discrete group $\Gamma$ on $A$ and $B$, respectively. 
Assume that there exist increasing sequences of finite self-adjoint subsets $\{F_n\}_{n=1}^\infty$ of $A$, 
$\{G_n\}_{n=1}^{\infty}$ of $B$, an increasing sequence of inverse-closed finite subsets 
$\{\Gamma_n\}_{n=1}^{\infty}$  of $\Gamma$, a decreasing sequence $\{\varepsilon_n\}_{n=1}^{\infty}$ of 
positive numbers and sequences of proper approximate cocycle morphisms 
$\{(\varphi_n, u_n)\}_{n=1}^{\infty}$ from $(A, \alpha)$ to $(B, \beta)$ 
and $\{(\psi_n, v_n)\}_{n=1}^{\infty}$ from $(B,\beta)$ to $(A, \alpha)$ 
such that for any $n\in\mathbb{N}$, 
\begin{align*}
(\mathrm{i}) &\; (\varphi_n, u_n) \text{ is  a proper } (\Gamma_n, F_n, \varepsilon_n) 
\text{-approximate cocycle morphism}, \\
(\mathrm{ii}) &\; (\psi_n, v_n) \text{ is  a proper } (\Gamma_n, G_n, \varepsilon_n) 
\text{-approximate cocycle morphism}, \\
(\mathrm{iii}) &\; \| \psi_n \circ \varphi_n (a) -a  \| <\varepsilon_n, \quad \forall a\in F_n, \\
(\mathrm{iv}) &\; \| \varphi_{n+1} \circ \psi_n (b) -b \| < \varepsilon_n, \quad \forall b\in G_n, \\
(\mathrm{v}) &\;  \| \psi_n\circ \varphi_n(a)(\psi_n^{\sim}(u_g^{(n)})v_g^{(n)}-1)\| 
< \varepsilon_n, \quad \forall a\in F_n,\; g\in 
\Gamma_n, \\
(\mathrm{vi}) &\; \| \varphi_{n+1}\circ \psi_n(b)
(\varphi_{n+1}^{\sim}(v_g^{(n)})u_g^{(n+1)}-1)\| < \varepsilon_n, \quad 
\forall b\in G_{n},\; g\in \Gamma_n, \\ 
(\mathrm{vii}) &\; \varphi_n(F_n)\subset G_n, \quad \psi_{n}(G_n) \subset F_{n+1}, \\ 
(\mathrm{viii}) &\; \alpha_g (F_n)\subset F_{n+1}, 
\quad \beta_g(G_n)\subset G_{n+1}, \quad \forall g\in \Gamma_{n+1}, 
\\
(\mathrm{ix}) &\; F_{n}v_g^{(n)} \subset F_{n+1}, \quad G_{n}u_g^{(n+1)} \subset G_{n+1}, \quad  
\forall g\in \Gamma_{n+1}, \\
(\mathrm{x}) &\; u_g^{(n)}\in \{b+\lambda 1_{B^{\sim}}\; |\; b\in G_n, \lambda\in\mathbb{C} \}, \quad 
\forall g\in \Gamma_n, \\
(\mathrm{xi}) &\; v_g^{(n)}\in \{a+\lambda 1_{A^{\sim}}\; |\; a\in F_{n+1}, \lambda\in\mathbb{C} \}, \quad 
\forall g\in \Gamma_{n+1}, \\
(\mathrm{xii}) &\; \overline{\bigcup_{n=1}^\infty F_n}=A, \quad \overline{\bigcup_{n=1}^\infty G_n}=B, \quad 
\bigcup_{n=1}^\infty \Gamma_n=\Gamma, \\
(\mathrm{xiii}) &\; \sum_{n=1}^{\infty}\varepsilon_n < \infty. 
\end{align*}
Then $\alpha$ is cocycle conjugate to $\beta$. 
\end{pro}
\begin{proof}
By Elliott's approximate intertwining argument (see, for example, \cite[Proposition 2.3.2]{Ror1}), 
we see that there exist isomorphisms $\varphi: A\to B$ and 
$\psi: B\to A$ such that
$$
\varphi(a)=\lim_{n\to \infty}\varphi_n(a) \quad \text{and}\quad \psi(b)= \lim_{n\to \infty}\psi_n(b)
$$
for any $a\in A$ and $b\in B$. Note that we have 
$$
\|\varphi(a_1a_2)-\varphi(a_1)\varphi(a_2)\|=\lim_{n\to\infty} \| \varphi_n(a_1a_2)-\varphi_n(a_1)
\varphi_n(a_2)\| \leq \lim_{n\to \infty} \varepsilon_n=0
$$
for any $a_1, a_2\in\bigcup_{n=1}^\infty F_n$, and this shows that $\varphi$ is a homomorphism by 
(xii). 
We shall show that $\{u_g^{(n)}\}_{n\in\mathbb{N}}$ is a strict Cauchy sequence of unitaries in 
$B^{\sim}$. 
Fix $g_0\in \Gamma$. Then there exists a natural number $n_0$ such that $g_0\in \Gamma_{n_0}$ 
by (xii).  
Let $k_0 \geq n_0$. 
For any $b\in G_{k_0}$, $g\in \Gamma_{n_0}$ and $k> k_0$, we have  
\begin{align*}
\| b(u_g^{(k+1)}-u_{g}^{(k)}) \| 
& = \|b(1-u_{g}^{(k)}u_{g}^{(k+1)*})\| \\ 
& = \| b-\varphi_{k+1}\circ \psi_{k}(bu_g^{(k)})u_g^{(k+1)*} \\
& \quad +
 (\varphi_{k+1}\circ \psi_{k}(bu_g^{(k)})-bu_{g}^{(k)})u_g^{(k+1)*}\| \\
(\mathrm{iv}), (\mathrm{ix})\; 
& < \| b-\varphi_{k+1}\circ \psi_{k}(bu_g^{(k)})u_g^{(k+1)*}\| + \varepsilon_{k} \\
& = \| b- \varphi_{k+1}\circ\psi_{k}(bu_{g}^{(k)})\varphi_{k+1}^{\sim}(v_g^{(k)}) \\
& \quad +  \varphi_{k+1}\circ \psi_{k}(bu_{g}^{(k)})(\varphi_{k+1}^{\sim}(v_g^{(k)})-u_{g}^{(k+1)*}) \| 
+ \varepsilon_k
\\
(\mathrm{vi}), (\mathrm{ix}) \;  &< 
\| b- \varphi_{k+1}\circ \psi_{k}(bu_{g}^{(k)})\varphi_{k+1}^{\sim}(v_g^{(k)})\|+2\varepsilon_k \\
& = \| b- \varphi_{k+1}(\psi_{k}(bu_g^{(k)})v_g^{(k)}) \\
& \quad +\varphi_{k+1}(\psi_{k}(bu_g^{(k)})v_g^{(k)})
-\varphi_{k+1}(\psi_{k}(bu_g^{(k)}))\varphi_{k+1}^{\sim}(v_g^{(k)})\| +2\varepsilon_k \\
(\mathrm{i}), (\mathrm{vii}), (\mathrm{ix}), (\mathrm{xi}) \; 
& < \| b- \varphi_{k+1}(\psi_{k}(bu_g^{(k)})v_g^{(k)}) \| +2\varepsilon_k+ \varepsilon_{k+1} \\
& =\| b- \varphi_{k+1}(\psi_{k}(b)\psi_k^{\sim}(u_g^{(k)})v_g^{(k)}) \\
& \quad + \varphi_{k+1}((\psi_{k}(b)\psi_k^{\sim}(u_g^{(k)})-\psi_{k}(bu_g^{(k)}))v_g^{(k)})\|
+2\varepsilon_k+ \varepsilon_{k+1} \\
(\mathrm{ii}), (\mathrm{x}) \; 
& < \| b- \varphi_{k+1}(\psi_{k}(b)\psi_k^{\sim}(u_g^{(k)})v_g^{(k)})\| +3\varepsilon_k+ \varepsilon_{k+1} \\
& =\| b- \varphi_{k+1}(\psi_k(\varphi_k\circ \psi_{k-1}(b))\psi_k^{\sim}(u_g^{(k)})v_g^{(k)})       \\
& \quad +  \varphi_{k+1}(\psi_k(\varphi_k\circ \psi_{k-1}(b)-b)\psi_k^{\sim}(u_g^{(k)})v_g^{(k)})\| 
+3\varepsilon_k+ \varepsilon_{k+1} \\
(\mathrm{iv}) \;
& <\| b- \varphi_{k+1}(\psi_k(\varphi_k\circ \psi_{k-1}(b))\psi_k^{\sim}(u_g^{(k)})v_g^{(k)})\| \\
& \quad +\varepsilon_{k-1}+3\varepsilon_k+ \varepsilon_{k+1} \\
& = \| b- \varphi_{k+1}(\psi_{k}\circ \varphi_k(\psi_{k-1}(b))) \\
& \quad + \varphi_{k+1}(\psi_{k}\circ \varphi_k(\psi_{k-1}(b))(1-\psi_k^{\sim}(u_g^{(k)})v_g^{(k)}))\| \\
& \quad +\varepsilon_{k-1}+3\varepsilon_k+ \varepsilon_{k+1} \\
(\mathrm{v}),(\mathrm{vii}) \; 
& < \| b- \varphi_{k+1}(\psi_{k}\circ \varphi_k(\psi_{k-1}(b)))\|
+\varepsilon_{k-1} +4\varepsilon_k+ \varepsilon_{k+1} \\
(\mathrm{iv})  \; 
& <\| b- \varphi_{k+1}(\psi_{k}(b))\|
+2\varepsilon_{k-1}+4\varepsilon_k+ \varepsilon_{k+1} \\
(\mathrm{iv}) \; 
& < 2\varepsilon_{k-1}+5\varepsilon_k+ \varepsilon_{k+1}
\leq 8\varepsilon_{k-1}. 
\end{align*}
Hence we have 
$$
\| b(u_g^{(k+1)}-u_{g}^{(k)}) \| < 8\varepsilon_{k-1} \eqno{(5.2.1)}
$$
for any $b\in G_{k_0}$, $g\in\Gamma_{n_0}$ and $k>k_0$. 
Since $G_{k_0}$ is self-adjoint and $\Gamma_{n_0}$ is inverse-closed, for any $b\in G_{k_0}$, 
$g\in \Gamma_{n_0}$ and $k> k_0+1$, we have 
\begin{align*}
\| b^*(u_{g}^{(k)*}-\beta_g(u_{g^{-1}}^{(k)})) \| 
& = \| b^*(1 -\beta_g(u_{g^{-1}}^{(k)})u_{g}^{(k)})\| 
= \|\beta_{g^{-1}}(b^*)(1-u_{g^{-1}}^{(k)}\beta_{g^{-1}}(u_{g}^{(k)})) \| \\
&= \|\beta_{g^{-1}}(b^*)(u_{g^{-1}g}^{(k)}-u_{g^{-1}}^{(k)}\beta_{g^{-1}}(u_{g}^{(k)})) \| \\
(\mathrm{iv}), (\mathrm{viii}) \; 
&< \|\varphi_{k}\circ \psi_{k-1}(\beta_{g^{-1}}(b^*)) 
(u_{g^{-1}g}^{(k)}-u_{g^{-1}}^{(k)}\beta_{g^{-1}}(u_{g}^{(k)})) \| +2\varepsilon_{k-1}  \\
(\mathrm{i}), (\mathrm{vii}), (\mathrm{viii}) \; 
& < 2\varepsilon_{k-1}+\varepsilon_{k}.
\end{align*}
Therefore for any $b\in G_{k_0}$, $g\in \Gamma_0$ and  $k> k_0+1$,  we have 
\begin{align*}
\| (u_g^{(k+1)}-u_{g}^{(k)})b\| 
& = \|b^* (u_{g}^{(k+1)*}-u_{g}^{(k)*})\| \\
& < \|b^*(\beta_g(u_{g^{-1}}^{(k+1)})-\beta_g(u_{g^{-1}}^{(k)})) \| 
+2\varepsilon_{k}+\varepsilon_{k+1}+2\varepsilon_{k-1}+\varepsilon_{k} \\
& =\| \beta_{g^{-1}}(b^*)(u_{g^{-1}}^{(k+1)}-u_{g^{-1}}^{(k)}) \| +\varepsilon_{k+1}+3\varepsilon_{k}
+2\varepsilon_{k-1}
\\
(5.2.1), (\mathrm{viii}) \; 
& < \varepsilon_{k+1}+3\varepsilon_{k}+10\varepsilon_{k-1}\leq 14\varepsilon_{k-1}.
\end{align*}
Consequently, we have 
$$
\| b(u_{g_0}^{(k+1)}-u_{g_0}^{(k)}) \| + \| (u_{g_0}^{(k+1)}-u_{g_0}^{(k)})b\| < 22\varepsilon_{k-1}
$$
for any $b\in G_{k_0}$ and $k>k_0+1$. Hence we see that $\{u_{g_0}^{(n)}\}_{n\in\mathbb{N}}$ 
is a strict Cauchy sequence of unitaries in $B^{\sim}$ by (xii) and (xiii). (Note that 
$\{u_{g_0}^{(n)}\}_{n\in\mathbb{N}}$ is a norm bounded sequence.) 
Therefore for any $g\in \Gamma$, 
there exists a unitary element $u_g$ in $M(B)$ such that $\{u_g^{(n)}\}_{n\in\mathbb{N}}$ 
converges strictly to $u_g$ because $M(B)$ is strictly complete. 
Since $B$ is an essential ideal in $M(B)$ and we have 
$$
\|\varphi(a)(u_{gh}- u_{g}\beta_g(u_{h}))\|=
\lim_{n\to \infty}\|\varphi_n(a)(u_{gh}^{(n)}- u_{g}^{(n)}\beta_g(u_{h}^{(n)}))\|
\leq \lim_{n\to \infty} \varepsilon_n=0
$$
for any $a\in \bigcup_{n=1}^\infty F_n$ and $g,h\in \Gamma$, 
we see that $u$ is a $\beta$-cocycle. 
Also, we have 
\begin{align*}
\| \varphi \circ \alpha_g (a) - \mathrm{Ad}(u_g)\circ \beta_g \circ \varphi (a) \|
& =\|\varphi (\alpha_g(a))- u_g\beta_g(\varphi (a))u_g^* \| \\
&= \lim_{n\to\infty} \|\varphi_n (\alpha_g(a))- u_g^{(n)}\beta_g(\varphi_n (a))u_g^{(n)*} \| \\
&\leq \lim_{n\to\infty}\varepsilon_n =0 
\end{align*}
for any $a\in\bigcup_{n=1}^\infty F_n$ and $g\in \Gamma$. This shows that $\alpha$ is cocycle 
conjugate to $\beta$ by (xii). 
\end{proof}

In \cite{Sza7}, Szab\'o defined the notions of approximate unitary equivalence (or approximate 
unitary conjugacy) and approximate innerness for cocycle morphisms. 
We shall consider a similar notion of approximate innerness  by using ultrapowers. 
A pair $(\Psi, U)$ is said to be a \textit{sequential asymptotic cocycle morphism from 
$(A,\alpha)$ to $(B, \beta)$} if $\Psi$ is a 
homomorphism from $A$ to $B^{\omega}$ and $U$ is a map from $\Gamma$ to the unitary group of 
$(B^{\sim})^{\omega}$ such that
$$
\Psi \circ \alpha_g(a) = \mathrm{Ad}(U_g)\circ \beta_g\circ \Psi(a) \quad  
\text{and} 
\quad 
\Psi(a)U_{gh}=\Psi (a)U_g\beta_g(U_h)
$$
for any $a\in A$ and $g,h\in \Gamma$. 
In the case where $A$ and $B$ are monotracial, a sequential asymptotic cocycle morphism 
$(\Psi, U)$ from $(A, \alpha)$ to $(B, \beta)$ is said to be 
\textit{trace preserving} if $\tau_{A}=\tau_{B, \omega}\circ \Psi$. 
A \textit{sequential asymptotic cocycle endomorphism of $(A, \alpha)$} is 
a sequential asymptotic cocycle morphism from $(A, \alpha)$ to $(A, \alpha)$. 
We say that a sequential asymptotic cocycle endomorphism $(\Psi, U)$ of $(A, \alpha)$ is \textit{inner} 
if there exists a unitary element $W$ in $(A^{\sim})^{\omega}$ such that 
$$
\Psi (a)=WaW^* \quad \text{and} \quad \Psi(a)U_g = \Psi(a) W\alpha_g (W^*)
$$
for any $a\in A$ and $g\in\Gamma$.

\section{Uniqueness theorem}\label{sec:unique}

Let $\gamma$ be a strongly outer action of a countable discrete amenable group $\Gamma$ on 
$\mathcal{W}$.
Assume that $F(\mathcal{W})^{\gamma}$ satisfies the following properties: \ \\
\ \\
(i) for any $\theta\in [0,1]$, there exists a projection $p$ in 
$F(\mathcal{W})^{\gamma}$ such that $\tau_{\mathcal{W}, \omega}(p)=\theta$, \ \\
(ii) if $p$ and $q$ are projections in $F(\mathcal{W})^{\gamma}$ 
such that $0<\tau_{\mathcal{W}, \omega}(p)=\tau_{\mathcal{W}, \omega}(q)$, 
then  $p$ is Murray-von Neumann equivalent to $q$.
\ \\
\ \\
Note that the properties above are a part of properties in the main theorem (Theorem 
\ref{thm:main2}) in this paper. 
In this section we shall show that every trace preserving sequential asymptotic cocycle endomorphism 
of $(\mathcal{W}, \gamma)$ is inner. 

Let $(\Psi, U)$ be a trace preserving sequential asymptotic cocycle endomorphism of 
$(\mathcal{W}, \gamma)$. 
By the Choi-Effros lifting theorem, there exists a sequence $\{\psi_n\}_{n\in\mathbb{N}}$ 
of c.c.p. maps from $\mathcal{W}$ to $\mathcal{W}$ such that $\Psi (a)= (\psi_n(a))_n$ for any 
$a\in\mathcal{W}$.  Define a homomorphism from $\mathcal{W}$ to $M_2(\mathcal{W})^{\omega}$ by 
$$
\Phi (a):= \left(\left(\begin{array}{cc}
               a   &   0    \\ 
               0   &  \psi_n (a)    
 \end{array} \right)\right)_n
$$
for any $a\in \mathcal{W}$. It is easy to see that 
$\tau_{M_2(\mathcal{W}), \omega}\circ \Phi= \tau_{\mathcal{W}}$ 
since we have $\tau_{\mathcal{W}, \omega}\circ \Psi = \tau_{\mathcal{W}}$. 
For any $g\in \Gamma$, there exists a sequence $\{u_{g, n}\}_{n\in\mathbb{N}}$ of unitary elements 
in $\mathcal{W}^{\sim}$ such that $U_g=(u_{g,n})_n$ by definition. 
Note that we have 
$$
(\psi_n \circ \gamma_g (a))_n = (\mathrm{Ad}(u_{g,n})\circ \gamma_g \circ \psi_n(a))_n
$$
and
$$
(\psi_n(a) u_{gh, n})_n= (\psi_n(a) u_{g,n}\gamma_g(u_{h,n}))_n
$$
for any $a\in \mathcal{W}$ and $g,h\in\Gamma$. 
For any $g\in\Gamma$ and 
$n\in\mathbb{N}$, define an automorphism $\beta_{g, n}$ of $M_2(\mathcal{W})$ by 
$$
\beta_{g, n}:= \mathrm{Ad}\left(\left(\begin{array}{cc}
               1                                &         0    \\ 
               0                                &        u_{g,n}
 \end{array} \right)\right)\circ \gamma_{g}\otimes\mathrm{id}_{M_2(\mathbb{C})}.
$$
If $(x_n)_n\in \Phi (\mathcal{W})$, then 
$(\beta_{g,n}(x_n))_n\in \Phi (\mathcal{W})$ for any $g\in\Gamma$ 
because we have 
\begin{align*}
\left(\beta_{g,n} \left( \left(\begin{array}{cc}
                       a   &        0    \\ 
                       0   &   \psi_n (a)    
 \end{array} \right)\right)\right)_n
&=\left( \left(\begin{array}{cc}
               \gamma_g(a)   &   0    \\ 
                       0           &  \mathrm{Ad}(u_{g,n})\circ \gamma_g\circ \psi_n (a)    
\end{array} \right)\right)_n \\
&=  \left( \left(\begin{array}{cc}
               \gamma_g(a)   &   0    \\ 
                       0           &  \psi_n ( \gamma_g(a))   
\end{array} \right)\right)_n
\end{align*}
for any $a\in\mathcal{W}$.  
Hence for any $g\in\Gamma$, we can define an automorphism $\beta_g^{\omega}$ of 
$F(\Phi(\mathcal{W}), M_2(\mathcal{W}))$ by 
$\beta_g^{\omega}([(x_n)_n])= [(\beta_{g,n}(x_n))_n]$ for any 
$[(x_n)_n]\in F(\Phi(\mathcal{W}), M_2(\mathcal{W}))$. 
For any 
$\left(\left(\begin{array}{cc}
                       a_n   &        b_n    \\ 
                       c_n   &        d_n    
 \end{array} \right)\right)_n\in M_2(\mathcal{W})^{\omega}\cap \Phi (\mathcal{W})^{\prime}$ 
and $a\in\mathcal{W}$, we have 
\begin{align*}
&\left(\begin{array}{cc}
                       a   &        0    \\ 
                       0   &      \Psi (a)    
 \end{array} \right)
\left(\beta_{gh,n} \left( \left(\begin{array}{cc}
                       a_n   &        b_n    \\ 
                       c_n   &        d_n    
 \end{array} \right)\right)\right)_n \\
&= 
\left( \left(\begin{array}{cc}
                       a\gamma_{gh}(a_n)   &        a\gamma_{gh}(b_n)u_{gh, n}^*    \\ 
                       \psi_n(a)u_{gh, n}\gamma_{gh}(c_n)   &        \psi_n(a)u_{gh, h}\gamma_{gh}(d_n)u_{gh, n}^*    
 \end{array} \right)\right)_n \\
&= 
\left( \left(\begin{array}{cc}
                       a\gamma_{gh}(a_n)   &        a\gamma_{gh}(b_n)u_{gh, n}^*    \\ 
                       \psi_n(a)u_{g, n}\gamma_{g}(u_{h, n}\gamma_{h}(c_n))  &      
                       \psi_n(a)u_{g, n}\gamma_{g}(u_{h, n}\gamma_{h}(d_n))u_{gh, n}^*    
 \end{array} \right)\right)_n \\
&= \left(\begin{array}{cc}
                       a   &        0    \\ 
                       0   &   \Psi (a)    
 \end{array} \right)
\left( \left(\begin{array}{cc}
                       \gamma_{gh}(a_n)   &        \gamma_{gh}(b_n)u_{gh, n}^*    \\ 
                       u_{g, n}\gamma_{g}(u_{h, n}\gamma_{h}(c_n))  &      
                       u_{g, n}\gamma_{g}(u_{h, n}\gamma_{h}(d_n))u_{gh, n}^*    
 \end{array} \right)\right)_n \\
&=
\left( \left(\begin{array}{cc}
                       \gamma_{gh}(a_n)   &        \gamma_{gh}(b_n)u_{gh, n}^*    \\ 
                       u_{g, n}\gamma_{g}(u_{h, n}\gamma_{h}(c_n))  &      
                       u_{g, n}\gamma_{g}(u_{h, n}\gamma_{h}(d_n))u_{gh, n}^*    
\end{array} \right)\right)_n
\left(\begin{array}{cc}
                       a   &        0    \\ 
                       0   &      \Psi (a)    
\end{array} \right) \\
&= 
\left( \left(\begin{array}{cc}
                       \gamma_{gh}(a_n)a   &        \gamma_{gh}(b_n)u_{gh, n}^*\psi_n(a)    \\ 
                       u_{g, n}\gamma_{g}(u_{h, n}\gamma_{h}(c_n))a  &      
                       u_{g, n}\gamma_{g}(u_{h, n}\gamma_{h}(d_n))u_{gh, n}^*\psi_n(a)    
\end{array} \right)\right)_n \\
&= 
\left( \left(\begin{array}{cc}
                       \gamma_{gh}(a_n)a   &        \gamma_{g}(\gamma_h(b_n)u_{h, n}^*)u_{g, n}^*\psi_n(a)    \\ 
                       u_{g, n}\gamma_{g}(u_{h, n}\gamma_{h}(c_n))a  &      
                       u_{g, n}\gamma_{g}(u_{h, n}\gamma_{h}(d_n)u_{h, n}^*)u_{g, n}^*\psi_n(a)    
\end{array} \right)\right)_n \\
&= 
\left( \left(\begin{array}{cc}
                       \gamma_{g}(\gamma_{h}(a_n))   &        \gamma_{g}(\gamma_h(b_n)u_{h,n}^*)u_{g, n}^*    \\ 
                       u_{g, n}\gamma_{g}(u_{h, n}\gamma_{h}(c_n))  &      
                       u_{g, n}\gamma_{g}(u_{h, n}\gamma_{h}(d_n)u_{h, n}^*)u_{g, n}^*    
\end{array} \right)\right)_n
\left(\begin{array}{cc}
                       a   &        0    \\ 
                       0   &   \Psi (a)    
\end{array} \right) \\
&=
\left(\begin{array}{cc}
                       a   &        0    \\ 
                       0   &    \Psi (a)    
 \end{array} \right)
\left(\beta_{g,n}\circ \beta_{h,n} \left( \left(\begin{array}{cc}
                       a_n   &        b_n    \\ 
                       c_n   &        d_n    
 \end{array} \right)\right)\right)_n.
\end{align*}
Therefore we see that $\beta^{\omega}:g\mapsto \beta^{\omega}_g$ is an action of
$\Gamma$ on $F(\Phi(\mathcal{W}), M_2(\mathcal{W}))$. 
Furthermore, we obtain the following proposition by 
Proposition \ref{pro:strict-comparison} and Proposition \ref{pro:relative-factor}. 

\begin{pro}\label{pro:strict-comparison-2-by-2}
If $a$ and $b$ are positive elements in 
$F(\Phi(\mathcal{W}), M_2(\mathcal{W}))^{\beta^{\omega}}$ satisfying
$d_{\tau_{M_2(\mathcal{W}), \omega}}(a)< d_{\tau_{M_2(\mathcal{W}), \omega}}(b)$, 
then there exists an element $r$ in $F(\Phi(\mathcal{W}), M_2(\mathcal{W}))^{\beta^{\omega}}$ 
such that $r^*br=a$. 
\end{pro}

Let $\{h_n\}_{n\in\mathbb{N}}$ be an approximate unit for  $\mathcal{W}$. 
Put 
$$
P:= \left[\left(\left(\begin{array}{cc}
                       h_n   &        0    \\ 
                       0   &        0    
\end{array} \right)\right)_n\right]\in F(\Phi(\mathcal{W}), M_2(\mathcal{W})).
$$
It is easy to see that $P$ is a projection in $F(\Phi(\mathcal{W}), M_2(\mathcal{W}))^{\beta^{\omega}}$ 
since $\{h_n^2\}_{n\in\mathbb{N}}$ and $\{\gamma_g(h_n)\}_{n\in\mathbb{N}}$ are  
also approximate units for  $\mathcal{W}$.
Furthermore, we see that 
if $\{t_n\}_{n=1}^{\infty}$ is a sequence of positive contractions in $\mathcal{W}$ such that 
$(t_n  a)_n= a$ for any $a\in\mathcal{W}$, then 
$$
P= \left[\left(\left(\begin{array}{cc}
                       t_n   &        0    \\ 
                       0   &        0    
\end{array} \right)\right)_n\right].
$$
Define a homomorphism $\iota_{11}$ from $F(\mathcal{W})^{\gamma}$ to 
$F(\Phi(\mathcal{W}), M_2(\mathcal{W}))^{\beta^{\omega}}$ by 
$$
\iota_{11}([(a_n)_n]):=  \left[\left(\left(\begin{array}{cc}
                       a_n   &        0    \\ 
                       0   &        0    
\end{array} \right)\right)_n\right]
$$
for any $[(a_n)_n]\in F(\mathcal{W})^{\gamma}$. It is easy to see that $\iota_{11}$ is a well-defined 
isomorphism from 
$F(\mathcal{W})^{\gamma}$ onto $PF(\Phi(\mathcal{W}), M_2(\mathcal{W}))^{\beta^{\omega}}P$. 
Let $\{F_n\}_{n\in\mathbb{N}}$ be an increasing sequence of finite subsets of $\mathcal{W}$ such that 
$\overline{\bigcup_{n=1}^{\infty}F_n}=\mathcal{W}$.
Since we have 
$$
(\psi_n(a)\psi_n(h_k))_n= (\psi_n(ah_k))_n
$$
for any $k\in\mathbb{N}$ and $a\in \mathcal{W}$, there exists 
a sequence $\{X_k\}_{k\in\mathbb{N}}$ in $\omega$ such that $\bigcap_{k=1}^{\infty}X_k=\emptyset$
and 
$$
\|\psi_n(a)\psi_n(h_k)- \psi_n(ah_k)\| < \dfrac{1}{k}
$$
for any $a\in F_k$ and $n\in X_k$. Set 
$$
k(n) := \left\{\begin{array}{cl}
0 & \text{if } n\notin X_1   \\
h_k & \text{if } n\in X_k\setminus \bigcup_{i=k+1}^{\infty}X_{i}\quad (k\in\mathbb{N})
\end{array}
\right.,
$$
then we have 
$$
(\psi_n(a)\psi_n(h_{k(n)}))_n=(\psi_n(ah_{k(n)}))_n=(\psi_n(a))_n
$$
for any $a\in\mathcal{W}$. Put 
$$
Q:= \left[\left(\left(\begin{array}{cc}
                       0   &            0       \\ 
                       0   &     \psi_n(h_{k(n)})    
\end{array} \right)\right)_n\right]\in F(\Phi(\mathcal{W}), M_2(\mathcal{W})).
$$
Since we have 
\begin{align*}
\Psi(a)(\psi_n(h_{k(n)})^2)_n
&= (\psi_n(a)\psi_n(h_{k(n)})\psi_n(h_{k(n)}))_n= 
(\psi_n(a)\psi_n(h_{k(n)}))_n \\
&=\Psi(a)
(\psi_n(h_{k(n)}))_n
\end{align*}
and
\begin{align*}
\Psi(a)(u_{g,n}\gamma_g(\psi_n(h_{k(n)}))u_{g,n}^*)_n 
&=(\psi_n(a)u_{g,n}\gamma_g(\psi_n(h_{k(n)}))u_{g,n}^*)_n \\ 
&= (u_{g,n}\gamma_g(\psi_n(\gamma_g^{-1}(a)))\gamma_g(\psi_n(h_{k(n)}))u_{g,n}^*)_n \\
&= (u_{g,n}\gamma_g(\psi_n(\gamma_g^{-1}(a))\psi_n(h_{k(n)}))u_{g,n}^*)_n \\
&=(u_{g,n}\gamma_g(\psi_n(\gamma_g^{-1}(a)))u_{g,n}^*)_n \\
&=(\psi_n(a)u_{g,n}u_{g,n}^*)_n=(\psi_n(a))_n =\Psi(a)(\psi_n(h_{k(n)}))_n
\end{align*}
for any $a\in \mathcal{W}$ and $g\in\Gamma$, 
we see that $Q$ is a projection in $F(\Phi(\mathcal{W}), M_2(\mathcal{W}))^{\beta^{\omega}}$. 
Furthermore, we see that 
if $\{t_n\}_{n=1}^{\infty}$ is a sequence of positive contractions in $\mathcal{W}$ such that 
$(t_n  a)_n= a$ and 
$(\psi_n(a)\psi_n(t_n))_n=(\psi_n(at_n))_n$ for any $a\in\mathcal{W}$, then 
$$
Q= \left[\left(\left(\begin{array}{cc}
                       0   &            0       \\ 
                       0   &     \psi_n(t_n)    
\end{array} \right) \right)_n\right].
$$
Let $[(a_n)_n]$ be an element in $F(\mathcal{W})^{\gamma}$. 
If a representative of $[(a_n)_n]$ satisfies 
$$
(\psi_n(a)\psi_n(\gamma_g(a_{n})))_n= (\psi_n(a\gamma_g(a_{n})))_n, \quad 
(\psi_n(\gamma_g(a_{n}))\psi_n(a))_n= (\psi_n(\gamma_g(a_{n})a))_n
$$
and 
$$
(u_{g,n}\gamma_g(\psi_n(a_n))u_{g, n}^*)_n=(\psi_n(\gamma_g(a_{n})))_n
$$
for any $a\in\mathcal{W}$ and $g\in \Gamma$, then we can define an element 
$$
\left[\left(\left(\begin{array}{cc}
                       0   &            0       \\ 
                       0   &         \psi_n(a_{n})  
\end{array} \right)\right)_n\right]
\in F(\Phi(\mathcal{W}), M_2(\mathcal{W}))^{\beta^{\omega}}.
$$
Indeed, we have 
\begin{align*}
\Psi(a)(\psi_n(a_n))_n
&= (\psi_n(a)\psi_n(a_n))_n=(\psi_n(aa_n))_n=(\psi_n(a_na))_n
=(\psi_n(a_n)\psi_n(a))_n \\
&= (\psi_n(a_n))_n\Psi(a)
\end{align*}
and 
\begin{align*}
\Psi(a)(u_{g,n}\gamma_g(\psi_n(a_n))u_{g, n}^*)_n
&= \Psi(a)(\psi_n(\gamma_g(a_{n})))_n= (\psi_n(a)\psi_n(\gamma_g(a_n)))_n \\
&=(\psi_n(a\gamma_g(a_n)))_n=(\psi_n(aa_n))_n=(\psi_n(a)\psi_n(a_n))_n \\
&=\Psi(a)(\psi_n(a_n))_n
\end{align*}
for any $a\in\mathcal{W}$ and $g\in \Gamma$.
In general, reindexing $[(a_n)_n]$ as above, we obtain an element $[(a_{k(n)})_n]$ in 
$F(\mathcal{W})^{\gamma}$ such that 
$$
(\psi_n(a)\psi_n(\gamma_g(a_{k(n)})))_n= (\psi_n(a\gamma_g(a_{k(n)})))_n, 
$$
$$
(\psi_n(\gamma_g(a_{k(n)}))\psi_n(a))_n= (\psi_n(\gamma_g(a_{k(n)})a))_n
$$
and 
$$
(u_{g,n}\gamma_g(\psi_n(a_{k(n)}))u_{g, n}^*)_n=(\psi_n(\gamma_g(a_{k(n)})))_n
$$
for any $a\in\mathcal{W}$ and $g\in \Gamma$ 
because we have 
$$
(\psi_n(a)\psi_n(\gamma_g(a_k)))_n= (\psi_n(a\gamma_g(a_k)))_n, \quad 
(\psi_n(\gamma_g(a_k))\psi_n(a))_n= (\psi_n(\gamma_g(a_k)a))_n
$$
and
$$
(u_{g,n}\gamma_g(\psi_n(a_k))u_{g, n}^*)_n= (\psi_n(\gamma_{g}(a_{k})))_n
$$
for any $a\in\mathcal{W}$, $g\in\Gamma$ and $k\in\mathbb{N}$. 
Furthermore, it can easily be checked that 
for any separable unital C$^*$-subalgebra $D$ 
of $F(\mathcal{W})^{\gamma}$, there exists a unital homomorphism 
$\iota_{22, D}$ from $D$ to $QF(\Phi(\mathcal{W}), M_2(\mathcal{W}))^{\beta^{\omega}}Q$. 
The following lemma is related to \cite[Lemma 6.2]{Na3} and \cite[Lemma 4.2]{Na4}. 

\begin{lem}\label{lem:2-by-2}
With notation as above, $P$ is Murray-von Neumann equivalent to $Q$ in 
$F(\Phi(\mathcal{W}), M_2(\mathcal{W}))^{\beta^{\omega}}$.
\end{lem}
\begin{proof}
For any $0<\varepsilon<1$, there exists a projection $q_{\varepsilon}$ 
in $F(\mathcal{W})^{\gamma}$ such that 
$\tau_{\mathcal{W}, \omega}(q_{\varepsilon})=1-\varepsilon$ by the assumption (ii) 
for $F(\mathcal{W})^{\gamma}$.
Reindexing a representative of $q_{\varepsilon}$, we obtain a projection $[(q_{\varepsilon, n})_n]$ in 
$F(\mathcal{W})^{\gamma}$ such that $\tau_{\mathcal{W}, \omega}([(q_{\varepsilon, n})_n])=1-\varepsilon$, 
$$
(\psi_n(a)\psi_n(\gamma_g(q_{\varepsilon, n})))_n= (\psi_n(a\gamma_g(q_{\varepsilon, n})))_n
$$
and 
$$
(u_{g,n}\gamma_g(\psi_n(q_{\varepsilon,n}))u_{g, n}^*)_n=(\psi_n(\gamma_g(q_{\varepsilon, n})))_n
$$
for any $a\in\mathcal{W}$ and $g\in \Gamma$.
Put 
$$
Q_{\varepsilon}:=\left[\left(\left(\begin{array}{cc}
                       0   &            0       \\ 
                       0   &         \psi_n(q_{\varepsilon, n})  
\end{array} \right)\right)_n\right]
\in F(\Phi(\mathcal{W}), M_2(\mathcal{W}))^{\beta^{\omega}},
$$
then $Q_\varepsilon$ is a projection such that 
$\tau_{M_2(\mathcal{W}), \omega}(Q_\varepsilon)=(1-\varepsilon)/2$. 
Proposition \ref{pro:strict-comparison-2-by-2} implies that there exists a contraction $R_{\varepsilon}$ in 
$F(\Phi(\mathcal{W}), M_2(\mathcal{W}))^{\beta^{\omega}}$ such that $R_{\varepsilon}^* P
R_{\varepsilon}= Q_{\varepsilon}$ since $\tau_{M_2(\mathcal{W}), \omega}(P)=1/2$. 
By the diagonal argument, there exist a projection $[(q^{\prime}_n)_n]$ in  $F(\mathcal{W})^{\gamma}$ 
and a contraction $R$ in $F(\Phi(\mathcal{W}), M_2(\mathcal{W}))^{\beta^{\omega}}$ such that 
$\tau_{\mathcal{W}, \omega}([(q_{n}^{\prime})_n])=1$, 
$$
(\psi_n(a)\psi_n(\gamma_g(q_{n}^{\prime})))_n= (\psi_n(a\gamma_g(q_{n}^{\prime})))_n, 
\quad 
(u_{g,n}\gamma_g(\psi_n(q_{n}^{\prime}))u_{g, n}^*)_n=(\psi_n(\gamma_g(q_{n}^{\prime})))_n
$$
for any $a\in\mathcal{W}$ and $g\in \Gamma$, and 
$$
R^* PR=\left[\left(\left(\begin{array}{cc}
                       0   &            0       \\ 
                       0   &         \psi_n(q_{n}^{\prime})  
\end{array} \right)\right)_n\right].
$$
By the assumption (i) for $F(\mathcal{W})^{\gamma}$, there exists an element 
$s=[(s_n)_n]$ in $F(\mathcal{W})^{\gamma}$ such that $s^*s=1$ and 
$ss^*= [(q_{n}^{\prime})_n]$. 
Reindexing representatives of $[(q_{n}^{\prime})_n]$, $[(s_n)_n]$, $R$, 
we may assume that 
$$
(\psi_n(a)\psi_n(\gamma_g(s_n)))_n= (\psi_n(a\gamma_g(s_n)))_n, 
\quad 
(\psi_n(\gamma_g(s_n))\psi_n(a))_n= (\psi_n(\gamma_g(s_n)a))_n, 
$$
and 
$$
(u_{g,n}\gamma_g(\psi_n(s_n))u_{g, n}^*)_n=(\psi_n(\gamma_g(s_n)))_n
$$
for any $a\in\mathcal{W}$ and $g\in \Gamma$, and we have 
$$
\left[\left(\left(\begin{array}{cc}
                       0   &            0       \\ 
                       0   &         \psi_n(s_{n})\psi_n(s_n^*)  
\end{array} \right)\right)_n\right]
=R^*PR, \quad 
\left[\left(\left(\begin{array}{cc}
                       0   &            0       \\ 
                       0   &         \psi_n(s_{n}^*)\psi_n(s_n)  
\end{array} \right)\right)_n\right]
=Q.
$$
Therefore $R^*PR$ is Murray-von Neumann equivalent to $Q$ in 
$F(\Phi(\mathcal{W}), M_2(\mathcal{W}))^{\beta^{\omega}}$. 
It is easy to see that there exists a projection $p$ in $F(\mathcal{W})^{\gamma}$ such that 
$\iota_{11}(p)=PRR^*P$ and $\tau_{\mathcal{W}, \omega}(p)=1$. 
By the assumption (i) for $F(\mathcal{W})^{\gamma}$, there exists an element in $r$ in 
$F(\mathcal{W})^{\gamma}$ such that $r^*r=1$ and $rr^*=p$. 
Since we have $\iota_{11}(r)^*\iota_{11}(r)=P$ and $\iota_{11}(r)\iota_{11}(r)^*=PRR^*P$, 
$P$ is Murray-von Neumann equivalent to 
$PRR^*P$ in $F(\Phi(\mathcal{W}), M_2(\mathcal{W}))^{\beta^{\omega}}$.
Consequently, $P$ is Murray-von Neumann equivalent to $Q$ in 
$F(\Phi(\mathcal{W}), M_2(\mathcal{W}))^{\beta^{\omega}}$.
\end{proof}

The following theorem is the Main theorem in this section. 

\begin{thm}
Let $\gamma$ be a strongly outer action of a countable discrete amenable group $\Gamma$ on 
$\mathcal{W}$.
Assume that $F(\mathcal{W})^{\gamma}$ satisfies the following properties: \ \\
(i) for any $\theta\in [0,1]$, there exists a projection $p$ in 
$F(\mathcal{W})^{\gamma}$ such that $\tau_{\mathcal{W}, \omega}(p)=\theta$, \ \\
(ii) if $p$ and $q$ are projections in $F(\mathcal{W})^{\gamma}$ 
such that $0<\tau_{\mathcal{W}, \omega}(p)=\tau_{\mathcal{W}, \omega}(q)$, 
then  $p$ is Murray-von Neumann equivalent to $q$.
\ \\
Then every trace preserving sequential asymptotic cocycle endomorphism $(\Psi, U)$ 
of $(\mathcal{W}, \gamma)$ is inner. 
\end{thm}
\begin{proof}
Let $\{h_n\}_{n\in\mathbb{N}}$ be an approximate unit for $\mathcal{W}$ satisfying 
$(\psi_n(a)\psi_n(h_n))_n= (\psi_n(ah_n))_n$ for any $a\in\mathcal{W}$. 
By Lemma \ref{lem:2-by-2}, there exists a partial isometry $V$ in  
$F(\Phi(\mathcal{W}), M_2(\mathcal{W}))^{\beta^{\omega}}$ such that 
$$
V^*V=\left[\left(\left(\begin{array}{cc}
                       h_n   &        0    \\ 
                       0   &        0    
\end{array} \right)\right)_n\right]
\quad \text{and} \quad 
VV^*=\left[\left(\left(\begin{array}{cc}
                       0   &        0    \\ 
                       0   &      \psi_n(h_n)    
\end{array} \right)\right)_n\right].
$$
It is easy to see that there exists an element $(v_n)_n \in \mathcal{W}^{\omega}$ such that 
$$
V=\left[\left(\left(\begin{array}{cc}
                       0    &      0    \\ 
                     v_n   &        0    
\end{array} \right)\right)_n\right]
$$
and we have 
$$
(v_n^*v_na)_n= (h_na)_n=a \quad \text{and} \quad 
(\psi_n(a) v_nv_n^{*})_n= (\psi_n(a)\psi_n(h_n))_n=(\psi_n(a))_n
$$
for any $a\in\mathcal{W}$. 
Since $V$ is an element in $F(\Phi(\mathcal{W}), M_2(\mathcal{W}))^{\beta^{\omega}}$, 
we have 
$$
(v_na)_n=(\psi_n(a)v_n)_n \quad \text{and} \quad (\psi_n(a)u_{g,n}\gamma_g(v_n))_n=
(\psi_n(a)v_n)_n
$$
for any $a\in\mathcal{W}$ and $g\in\Gamma$. These imply 
$$
\Psi(a)= (v_nav_n^*)_n \quad \text{and} \quad \Psi (a)U_g(\gamma_g(v_n)v_n^*)_n= \Psi(a)
$$
for any $a\in\mathcal{W}$ and $g\in\Gamma$. 
Since $\mathcal{W}$ has stable rank one, we may assume that $v_n$ is an invertible element in 
$\mathcal{W}^{\sim}$ for any $n\in\mathbb{N}$. For any $n\in\mathbb{N}$, put
$w_n:= v_n(v_n^*v_n)^{-1/2}$, and let $W:= (w_n)_n$. Then $W$ is a unitary element in 
$(\mathcal{W}^{\sim})^{\omega}$ and $(w_na)_n =(v_na)_n$ for any $a\in\mathcal{W}$. 
Furthermore, we have 
\begin{align*}
\Psi (a)=(v_nav_n^*)_n= (w_naw_n^*)=WaW^*
\end{align*}
and
\begin{align*}
\Psi(a)
&= \Psi (a)U_g(\gamma_g(v_n)v_n^*)_n = U_g\gamma_g(\Psi(\gamma_g^{-1}(a))(\gamma_g(v_n)v_n^*)_n \\
&=U_g(\gamma_g(\psi_n(\gamma_g^{-1}(a))v_n)v_n^*)_n=U_g(\gamma_g(v_n\gamma_g^{-1}(a))v_n^*)_n \\
&=U_g(\gamma_g(w_n\gamma_g^{-1}(a))v_n^*)_n=U_g(\gamma_g(w_n)av_n^*)_n
=U_g(\gamma_g(w_n)aw_n^*)_n \\
&= \Psi(a)U_g(\gamma_g(w_n)w_n^*)_n= \Psi(a) U_g\gamma_g(W)W^*
\end{align*}
for any $a\in\mathcal{W}$ and $g\in \Gamma$. Therefore $(\Psi, U)$ is inner. 
\end{proof}

We shall consider a generalization of proper approximate cocycle morphisms.
Let $\alpha$ and $\beta$ be actions of  $\Gamma$ on $A$ and $B$, 
respectively, and let $\Gamma_0\subset \Gamma$, $F\subset A$, $G\subset B$ and $\varepsilon>0$.  
A \textit{proper $(\Gamma_0, F, \varepsilon)$-approximate quasi cocycle morphism 
from $(A, \alpha)$ to $(B, \beta)$} is a pair $(\varphi, u)$, where $\varphi$ is 
an $(F, \varepsilon)$-multiplicative map from $A$ to $B$ and 
$u$ is a map from $\Gamma$ to contractions in  $B^{\sim}$
such that 
$$
\| \varphi (a)(u_{gh}-u_{g}\beta_g(u_{h})) \|< \varepsilon, \quad
\|u_g^*u_g-1\|< \varepsilon, \quad \|u_gu_g^* -1 \|< \varepsilon
$$
and 
$$
\| \varphi \circ \alpha_g (a) - \mathrm{Ad}(u_g)\circ \beta_g \circ \varphi (a) \| < \varepsilon
$$
for any $g\in \Gamma_0$ and $a\in F$. 
Assume that $\{\Gamma_n\}_{n=1}^{\infty}$ is an increasing sequence of subsets of $\Gamma$ 
with $\bigcup_{n=1}^{\infty}\Gamma_n=\Gamma$ and 
$\{F_n\}_{n=1}^{\infty}$ is an increasing sequence of subsets of $A$ with 
$\overline{\bigcup_{n=1}^{\infty}F_n}=A$. It is easy to see that if  
$(\varphi_n, u_g^{(n)})$ is a proper $(\Gamma_n, F_n, 1/n)$-approximate quasi cocycle morphism 
from $(A, \alpha)$ to $(B, \beta)$ for any $n\in\mathbb{N}$, then  a sequence 
$\{(\varphi_n, u_g^{(n)})\}_{n\in\mathbb{N}}$ induces a 
sequential asymptotic cocycle morphism from $(A,\alpha)$ to $(B, \beta)$. 
Hence we obtain the following corollary by the theorem above. 

\begin{cor}\label{cor:main-section6}
Let $\gamma$ be a strongly outer action of a countable discrete amenable group $\Gamma$ on 
$\mathcal{W}$.
Assume that $F(\mathcal{W})^{\gamma}$ satisfies the following properties: \ \\
(i) for any $\theta\in [0,1]$, there exists a projection $p$ in 
$F(\mathcal{W})^{\gamma}$ such that $\tau_{\mathcal{W}, \omega}(p)=\theta$, \ \\
(ii) if $p$ and $q$ are projections in $F(\mathcal{W})^{\gamma}$ 
such that $0<\tau_{\mathcal{W}, \omega}(p)=\tau_{\mathcal{W}, \omega}(q)$, 
then  $p$ is Murray-von Neumann equivalent to $q$.
\ \\
For any finite subsets $F\subset \mathcal{W}$, 
$\Gamma_0\subset \Gamma$ and $\varepsilon>0$, there exist finite subsets 
$F^{\prime} \subset \mathcal{W}$, $\Gamma_0^{\prime}\subset \Gamma$ and $\delta>0$ 
such that the following holds. If $(\varphi, u)$ is a proper 
$(\Gamma_0^{\prime}, F^{\prime}, \delta)$-approximate quasi cocycle morphism from 
$(\mathcal{W}, \gamma)$ to $(\mathcal{W}, \gamma)$ such that 
$$
| \tau_{\mathcal{W}}(\varphi (a)) -\tau_{\mathcal{W}}(a) | < \delta
$$
for any $a\in F^{\prime}$, then there exists a unitary element $w$ in $\mathcal{W}^{\sim}$ such that 
$$
\| \varphi(a) - waw^* \| < \varepsilon \quad \text{and} \quad 
\| \varphi (a) (u_g -w\gamma_g(w^*)) \| < \varepsilon
$$
for any $a\in F$ and $g\in \Gamma_0$.  
\end{cor}

\section{Existence theorem}\label{sec:exist}

Let $A$ be a simple separable nuclear monotracial C$^*$-algebra, and let 
$\alpha$ be a strongly outer action of a countable discrete amenable group 
$\Gamma$ on $A$.  
Assume that $A\rtimes_{\alpha}\Gamma$ is $\mathcal{W}$-\textit{embeddable} 
(see \cite{LN}), that is, 
there exists an injective homomorphism from $A\rtimes_{\alpha}\Gamma$ to $\mathcal{W}$. 
Note that if $A\rtimes_{\alpha}\Gamma$ is $\mathcal{W}$-embeddable, then 
there exists a trace preserving homomorphism from  $A\rtimes_{\alpha}\Gamma$ to $\mathcal{W}$ 
by \cite[Lemma 6.3]{Na4}. 
Let $B$ be a simple separable non-type I nuclear monotracial C$^*$-algebra with strict comparison 
and $B\subset\overline{\mathrm{GL}(B^{\sim})}$, 
and let $\beta$ be an outer action of $\Gamma$ on $B$. 
In this section we shall show that there exists a trace preserving 
sequential asymptotic cocycle morphism from $(A, \alpha)$ to $(B, \beta)$.

\begin{lem}\label{lem:exist}
Let $A$ and $B$ be simple separable monotracial C$^*$-algebras, and let 
$\alpha$ and $\beta$ be outer actions of a countable discrete group $\Gamma$ on 
$A$ and $B$, respectively. Assume that $B \subset \overline{\mathrm{GL}(B^{\sim})}$. 
If there exists a homomorphism $\Phi$ from $A\rtimes_{\alpha}\Gamma$ to 
$(B^{\omega})^{\beta}$ such that 
$\tau_{A}\circ E_{\alpha}= \tau_{B, \omega}\circ \Phi$ where $E_{\alpha}$ is the canonical conditional 
expectation from $A\rtimes_{\alpha}\Gamma$ onto $A$, then there exists 
a trace preserving sequential asymptotic cocycle morphism from $(A, \alpha)$ to $(B, \beta)$. 
\end{lem}
\begin{proof}
Let $\{h_n\}_{n\in\mathbb{N}}$ be an approximate unit for $A$. We denote by $\lambda_g$ the 
implementing unitary of $\alpha_g$ in $M(A\rtimes_{\alpha}\Gamma)$. 
For any finite subsets $F\subset A$, $\Gamma_0\subset \Gamma$ and $\varepsilon>0$, 
there exists a natural number $N$ such that 
$$
\| \Phi (\lambda_gh_{N})\Phi(a)\Phi (\lambda_gh_{N})^*-\Phi (\alpha_g(a)) \| < \varepsilon,
\quad 
\| (\Phi(\lambda_gh_{N})^*\Phi(\lambda_gh_{N})-1)\Phi (a)\| < \varepsilon
$$
and 
$$
\| (\Phi (\lambda_{gh}h_{N})- \Phi(\lambda_g h_{N})\Phi (\lambda_h h_{N}))\Phi(a)\| <\varepsilon
$$
for any $a\in F$ and $g,h\in \Gamma_0$. By the diagonal argument, we obtain a set 
$\{V_g\; |\; g\in\Gamma \}$ of elements in $(B^{\omega})^{\beta}$ such that
$$
V_g\Phi (a)V_g^* =\Phi (\alpha_g (a)), \quad V_g^*V_g\Phi(a)=\Phi(a) \quad \text{and} \quad 
V_{gh}\Phi(a)= V_{g}V_{h}\Phi(a)
$$
for any $a\in A$ and $g, h\in\Gamma$. 
For any $g\in \Gamma$, there exists a sequence $\{v_{g, n}\}_{n\in\mathbb{N}}$ of invertible elements 
in $B^{\sim}$ such that $V_g=(v_{g, n})_n$ since $B \subset \overline{\mathrm{GL}(B^{\sim})}$. 
For any $g\in \Gamma$, let $U_g:= (v_{g, n}(v_{g, n}^*v_{g,n})^{-1/2})_n$. Then $U_g$ is a unitary element 
in $(B^{\sim})^{\omega}$ and $U_g\Phi (a)= V_g \Phi(a)$ for any $g\in\Gamma$ and $a\in A$. 
For any $a\in A$ and $g\in\Gamma$, we have 
$$
U_g\beta_g(\Phi(a))U_g^*= U_g\Phi (a) U_g^*= V_g\Phi(a)V_g^*=\Phi (\alpha_g (a)).
$$
Note that we also have $U_g\Phi(a)=\Phi(\alpha_g(a))U_g$. 
For any $a\in A$ and $g,h\in\Gamma$, we have 
\begin{align*}
\Phi(a)U_{gh}
&=U_{gh}\Phi(\alpha_{h^{-1}g^{-1}}(a))=V_{gh}\Phi (\alpha_{h^{-1}g^{-1}}(a))
=V_{g}V_{h}\Phi (\alpha_{h^{-1}g^{-1}}(a)) \\
&= V_{g}\beta_g(V_h\Phi (\alpha_{h^{-1}g^{-1}}(a)))
=V_{g}\beta_g(U_h\Phi (\alpha_{h^{-1}g^{-1}}(a))) \\
&=V_g\beta_g(\Phi(\alpha_{g^{-1}}(a))U_h) 
= V_g \Phi(\alpha_{g^{-1}}(a))\beta_g(U_h) \\
&= U_g\Phi(\alpha_{g^{-1}}(a))\beta_g(U_h) 
=\Phi(a)U_g\beta_g(U_h).
\end{align*}
Therefore $(\Phi|_A, U)$ is a trace preserving 
sequential asymptotic cocycle morphism from $(A, \alpha)$ to $(B, \beta)$. 
\end{proof}

By the lemma above, we would like to show that there exists a trace preserving homomorphism 
from $\mathcal{W}$ to $(B^{\omega})^{\beta}$. 
Actually, we shall show that there exists a trace preserving homomorphism 
from $\mathcal{W}$ to $(B^{\omega}\cap B^{\prime})^{\beta}$ because 
$(B^{\omega}\cap B^{\prime})^{\beta}$ has good properties rather than $(B^{\omega})^{\beta}$. 
Note that the proof of this result is based on 
Schafhauser's ideas \cite{Sc} in his proof of the Tikuisis-White-Winter theorem \cite{TWW}. 
See also \cite{Sc2} and \cite[Section 5]{Na4}. 
We say that an extension $0\longrightarrow I \longrightarrow C \longrightarrow A\longrightarrow 0$ 
is \textit{purely large} if for any $x\in C\setminus  I$, $\overline{xIx^*}$ contains a stable 
C$^*$-subalgebra which is full in $I$. We refer the reader to \cite{EllK} and \cite{G} for details of 
purely large extensions. By Proposition \ref{pro:equiv-surjective}, 
there exists the following extension: 
$$
\xymatrix{
\eta: & 0 \ar[r]  & \mathrm{ker}\; \varrho|_{(B^{\omega}\cap B^{\prime})^{\beta}} \ar[r] & 
(B^{\omega}\cap B^{\prime})^{\beta}  \ar[r] & M_{\omega}^{\tilde{\beta}} \ar[r] & 0
}
$$
where $M=\pi_{\tau_{B}}(B)^{''}$. Put 
$J:=\mathrm{ker}\; \varrho|_{(B^{\omega}\cap B^{\prime})^{\beta}}$. Note that 
$J=\{x\in (B^{\omega}\cap B^{\prime})^{\beta}\; | \; \tau_{B, \omega}(x^*x)=0\}$ and for 
any $b\in 
(B^{\omega}\cap B^{\prime})^{\beta}_{+}$,  $b\in J$ if and only if 
$\tau_{B, \omega}(b)=0$. 
Since $M_{\omega}^{\tilde{\beta}}$ is a II$_1$ factor by Proposition \ref{pro:MS-lemma-4.1}, 
there exists a trace preserving homomorphism $\Pi$ from $\mathcal{W}$ to 
$M_{\omega}^{\tilde{\beta}}$. 
Consider a pullback extension 
$$
\xymatrix{
\Pi^*\eta: &   0 \ar[r] 
&J  \ar[r]\ar@{=}[d] & E \ar[r]^{\hat{\varrho}}\ar[d]^{\hat{\Pi}} 
& \mathcal{W} \ar[r]\ar[d]^{\Pi} & 0 \\
\eta: & 0 \ar[r]  & J \ar[r] & (B^{\omega}\cap B^{\prime})^{\beta} 
\ar[r] & M_{\omega}^{\tilde{\beta}} \ar[r] & 0
}
$$
where $E=\{(a,x)\in \mathcal{W}\oplus (B^{\omega}\cap B^{\prime})^{\beta} \; |\; 
\Pi(a)=\varrho(x)\}$, $\hat{\varrho}((a,x))=a$ and $\hat{\Pi}((a,x))=x$ for any $(a,x)\in E$. 
Using Blackadar's technique (see \cite[II.8.5]{Bla}), we shall construct 
a separable C$^*$-subalgebras $B_0\subset  (B^{\omega}\cap B^{\prime})^{\beta},$ 
$J_0\subset J$ and  $M_0\subset M_{\omega}^{\tilde{\beta}}$ such that 
$\Pi(\mathcal{W})\subset M_0$ and 
$$
\xymatrix{
\eta_0: & 0 \ar[r]  & J_0 \ar[r]& B_0  \ar[r]^{\varrho|_{B_0}} \ar[r] & M_0 \ar[r]& 0.
}
$$
is a purely large extension as in \cite[Section 5]{Na4}. 

Since $M_{\omega}^{\tilde{\beta}}$ is a factor, the following proposition is an immediate consequence 
of Proposition \ref{pro:strict-comparison-notk}. 

\begin{pro}\label{pro:strict-comparison-7}
If $a$ and $b$ are positive elements in $(B^{\omega}\cap B^{\prime})^{\beta}$ 
satisfying $d_{\tau_{B, \omega}}(a)<d_{\tau_{B, \omega}}(b)$, then there exists an element $r$ in 
$(B^{\omega}\cap B^{\prime})^{\beta}$ such that $r^*br=a$. 
\end{pro}

\begin{lem}
With notation as above, let $b$ be a positive element in 
$(B^{\omega}\cap B^{\prime})^{\beta}\setminus J$. 
Then there exists a positive contraction $f$ in 
$\overline{b (B^{\omega}\cap B^{\prime})^{\beta}b}$ such that 
$$
\inf_{m\in\mathbb{N}}\tau_{B, \omega}(f^m)>0.
$$ 
\end{lem}
\begin{proof}
We may assume that $b$ is a contraction. Then we have 
$d_{\tau_{B, \omega}}(b)\geq \tau_{B, \omega}(b)>0$. Since $M_{\omega}^{\tilde{\beta}}$ is a 
II$_1$ factor, there exists a projection $p$ in $M_{\omega}^{\tilde{\beta}}$ such that 
$\tilde{\tau}(p)=d_{\tau_{B, \omega}}(b)/2>0$ where $\tilde{\tau}$ is the unique tracial state on 
$M_{\omega}^{\tilde{\beta}}$. 
Note that we have 
$\tau_{B, \omega}= \tilde{\tau}\circ \varrho|_{(B^{\omega}\cap B^{\prime})^{\beta}}$. 
By surjectivity of 
$\varrho|_{(B^{\omega}\cap B^{\prime})^{\beta}}$, there exists a positive contraction 
$e$ in $(B^{\omega}\cap B^{\prime})^{\beta}$ such that $\varrho (e)=p$. 
Since we have 
$$
d_{\tau_{B, \omega}}(e)=\lim_{m\to\infty}\tau_{B, \omega}(e^{\frac{1}{m}})
=\lim_{m\to\infty}\tilde{\tau}(\varrho (e^{\frac{1}{m}}))
=\lim_{m\to\infty}\tilde{\tau}(p)=\frac{d_{\tau_{B, \omega}}(b)}{2},
$$  
Proposition \ref{pro:strict-comparison-7} implies that there exists an element $r$ in 
$(B^{\omega}\cap B^{\prime})^{\beta}$ such that $r^*br=e$. Put $f:= b^{1/2}rr^*b^{1/2}$. 
Then $f$ is a positive contraction in $\overline{b (B^{\omega}\cap B^{\prime})^{\beta}b}$. 
Note that $\varrho(f)$ is a projection in  $M_{\omega}^{\tilde{\beta}}$ 
satisfying $\tilde{\tau}(\varrho(f))=\tilde{\tau}(p)>0$ 
because $\varrho(f)$ is Murray-von Neumann equivalent to $\varrho(e)=p$. 
Therefore we have 
\begin{align*}
\inf_{m\in\mathbb{N}}\tau_{B, \omega}(f^m)=\inf_{m\in\mathbb{N}}
\tilde{\tau}(\varrho (f^m))=\inf_{m\in\mathbb{N}}\tilde{\tau}(\varrho(f))=\tilde{\tau}(\varrho(f))>0.
\end{align*} 
\end{proof}

The following proposition is an equivariant version of \cite[Proposition 5.4]{Na4}. 
Note that we consider $(B^{\omega}\cap B^{\prime})^{\beta}$ rather than $(B^{\omega})^{\beta}$. 
\begin{pro}\label{pro:main-lem-blatech}
With notation as above, let $b$ be a positive element in 
$(B^{\omega}\cap B^{\prime})^{\beta}\setminus J$. \ \\
(i) For any positive element $a$ in $\overline{bJb}$, there exists a positive element $c$ in 
$\overline{bJb}$ such that $ac=0$ and $c$ is Murray-von Neumann equivalent to $a$ in 
$\overline{bJb}$. \ \\
(ii) For any positive element $a$ in $J$, there exist a positive element $d$ in $\overline{bJb}$ and 
an element $r$ in $J$ such that $r^*dr=a$. 
\end{pro}
\begin{proof}
(i) We may assume that $a$ is a contraction. By the lemma above, there exists a positive contraction 
$f$ in $\overline{b(B^{\omega}\cap B^{\prime})^{\beta}b}$ such that 
$\inf_{m\in\mathbb{N}}\tau_{B, \omega}(f^m)>0$. Since $f^{1/2}a$ is an element in $J$, there exists 
a positive contraction $e\in J$ such that $af^{1/2}e=eaf^{1/2}=af^{1/2}$ by \cite[Lemma 4.4]{MS2}. 
Put $\tilde{f}= f-f^{1/2}ef^{1/2}\in \overline{b(B^{\omega}\cap B^{\prime})^{\beta}b}_+$, then we have 
$a\tilde{f}=0$. 
Since we have $\varrho(\tilde{f}^{m})= \varrho (f^{m})$ for any $m\in\mathbb{N}$, 
$$
\inf_{m\in\mathbb{N}}\tau_{B}(\tilde{f}^{m})=\inf_{m\in\mathbb{N}}\tau_{B, \omega}(f^m)>0. 
$$
Therefore Theorem \ref{thm:si} implies that there exists an element $s$ in 
$(B^{\omega}\cap B^{\prime})^{\beta}$ such that $\tilde{f}s=s$ and $s^*s=a$. Put $c:=ss^*$, then 
$ac=ass^*=a\tilde{f}ss^*=0$. 
Since we have $\tilde{f}c\tilde{f}=c$, $c$ is an element in 
$\overline{b(B^{\omega}\cap B^{\prime})^{\beta}b}$. 
Furthermore, $c$ is an element in $\overline{bJb}$ and is Murray-von Neumann equivalent to 
$a$ in $\overline{bJb}$ because $\overline{bJb}$ is a closed ideal of 
$\overline{b(B^{\omega}\cap B^{\prime})^{\beta}b}$. 

(ii) We may assume that $b$ is a contraction. Since we have 
$d_{\tau_{B, \omega}}(b)\geq \tau_{B, \omega}(b)>0$ and 
$d_{\tau_{B, \omega}}(a^{1/5})=0$, there exists an element $t$ such that $t^*bt= a^{1/5}$ 
by Proposition \ref{pro:strict-comparison-7}. 
Put $d:= bta^{1/5}t^*b\in \overline{bJb}$ and $r:=ta^{1/5}\in J$, then we have $r^*dr= a$. 
\end{proof}

Using Proposition \ref{pro:main-lem-blatech}.(ii) instead of \cite[Proposition 5.4.(ii)]{Na4}, 
the same proof as \cite[Lemma 5.9]{Na4} (based on Blackadar's technique) shows the following 
lemma. 

\begin{lem} \label{lem:5.9}
With notation as above, let $\{b_k\; |\; k\in\mathbb{N}\}$ be a countable subset of 
$(B^{\omega}\cap B^{\prime})^{\beta}\setminus J$ and $S$ be a separable subset of 
$(B^{\omega}\cap B^{\prime})^{\beta}$. Then there exists a separable C$^*$-subalgebra $A$ of 
$(B^{\omega}\cap B^{\prime})^{\beta}$ such that
$\{b_k\; |\; k\in\mathbb{N}\}\cup S\subset A$ and $\overline{b_k(A\cap J)b_k}$ is full in 
$A\cap J$ for any $k\in\mathbb{N}$. 
\end{lem}

Using Proposition \ref{pro:main-lem-blatech} and Lemma \ref{lem:5.9} instead of 
\cite[Proposition 5.4]{Na4}, \cite[Lemma 5.8]{Na4} and \cite[Lemma 5.10]{Na4}, 
we can construct a purely large separable extension $\eta_0$ by a 
similar way as in \cite[Section 5]{Na4}. We shall give some details for the reader's convenience. 

Since $\varrho|_{(B^{\omega}\cap B^{\prime})^{\beta}}$ is surjective and $\mathcal{W}$ is 
separable, there exists a separable C$^*$-subalgebra $B_1$ of $(B^{\omega}\cap B^{\prime})^{\beta}$
such that $\Pi (\mathcal{W})\subset \varrho (B_1)$. 
By separability of $B_1$, there exist a countable dense subset 
$\{a_{1, m}\; | \; m\in\mathbb{N}\}$ of $(B_1\cap J)_{+}$ and a countable dense subset 
$\{b_{1,k}\; | \; k\in \mathbb{N} \}$ of $B_{1+}$. Put 
$$
T_1:=\{(k,l)\in\mathbb{N}\times\mathbb{N}\; |\; (b_{1,k}-1/l)_{+}\notin J\}.
$$
By Proposition \ref{pro:main-lem-blatech}, for any $(k,l)\in T_1$ and $m\in\mathbb{N}$,
there exist elements $c_{1,1,(k,l),m}$ and $z_{1,1,(k,l),m}$ in 
$\overline{(b_{1, k}-1/l)_{+}J(b_{1, k}-1/l)_{+}}$ such that 
$$
(b_{1,k}-1/l)_{+}a_{1,m}(b_{1,k}-1/l)_{+}c_{1,1,(k,l),m}=0, \quad 
z_{1,1,(k,l),m}z_{1,1,(k,l),m}^*=c_{1,1,(k,l),m}
$$
and
$$
z_{1,1,(k,l),m}^*z_{1,1,(k,l),m}=(b_{1,k}-1/l)_{+}a_{1,m}(b_{1,k}-1/l)_{+}.
$$
Lemma \ref{lem:5.9} implies that there exists a separable C$^*$-subalgebra $B_2$ of 
$(B^{\omega}\cap B^{\prime})^{\beta}$ such that 
$$
\{c_{1,1,(k,l),m}, z_{1,1,(k,l),m}\; |\; (k,l)\in T_1, m\in\mathbb{N}\}\cup B_1 \subset B_2
$$
and $\overline{(b_{1, k}-1/l)_{+}(B_2\cap J)(b_{1,k}-1/l)_{+}}$ is full in $B_2\cap J$ for any $(k,l)\in T_1$. 
Repeating this process, for any $n\in\mathbb{N}$, there exist a separable C$^*$-subalgebra 
$B_{n}$ of $(B^{\omega}\cap B^{\prime})^{\beta}$, countable dense subsets 
$\{a_{n, m}\;| \; m\in\mathbb{N}\}$ of $(B_{n}\cap J)_{+}$, 
$\{b_{n, m}\; |\; m\in\mathbb{N}\}$ of $B_{n+}$, a subset $T_{n}$ of 
$\mathbb{N}\times \mathbb{N}$ and a subset 
$\{c_{n,i,(k,l),m}, z_{n,i,(k,l),m}\; |\; 1\leq i\leq n, (k,l)\in T_i, m\in\mathbb{N}\}$ of $J$  such that 
$$
B_{n}\subset B_{n+1}, \quad T_{n}=\{(k,l)\in\mathbb{N}\times \mathbb{N}\; |\; (b_{n,k}-1/l)_+\notin J\},
$$
$$
c_{n,i,(k,l),m}, z_{n,i,(k,l),m}\in \overline{(b_{i,k}-1/l)_{+}(B_{n+1}\cap J)(b_{i,k}-1/l)_{+}},
$$
$$
(b_{i,k}-1/l)_{+}a_{n,m}(b_{i,k}-1/l)_{+}c_{n,i,(k,l),m}=0, \quad 
z_{n,i,(k,l),m}z_{n,i,(k,l),m}^*=c_{n,i,(k,l),m},
$$
$$
z_{n,i,(k,l),m}^*z_{n,i,(k,l),m}=(b_{i,k}-1/l)_{+}a_{n,m}(b_{i,k}-1/l)_{+}
$$
and 
$\overline{(b_{i,k}-1/l)_{+}(B_{n+1}\cap J)(b_{i, k}-1/l)_{+}}$ is full in 
$B_{n+1}\cap J$ for any $1\leq i\leq n$ and $(k,l)\in T_i$. 
Put 
$$
B_0:= \overline{\bigcup_{n=1}^{\infty}B_n}, \quad J_0:=B_0\cap J, \quad 
M_0 := \varrho (B_0)  
$$
and 
$$
\xymatrix{
\eta_0: & 0 \ar[r]  & J_0 \ar[r]& B_0  \ar[r]^{\varrho|_{B_0}} \ar[r] & M_0 \ar[r]& 0.
}
$$
Then $\eta_0$ is a separable extension and $\Pi (\mathcal{W})\subset M_0$. 
For any $i\in\mathbb{N}$ and $(k,l)\in T_{i}$, 
$\overline{(b_{i,k}-1/l)_{+}J_0(b_{i,k}-1/l)_{+}}$ is full in $J_0$ because we have 
$J_0=\overline{\bigcup_{n=1}^{\infty}(B_n\cap J)}$ and 
\begin{align*}
B_{n}\cap J 
&=\overline{B_{n}\cap J(b_{i,k}-1/l)_{+}(B_{n}\cap J)(b_{i,k}-1/l)_{+}B_{n}\cap J} \\
&\subset \overline{J_0(b_{i,k}-1/l)_{+}J_0(b_{i,k}-1/l)_{+}J_0} 
\end{align*}
for any $n> i$. Also, for any $n_0\in\mathbb{N}$, 
$\{a_{n, m}\; |\; n, m\in\mathbb{N}, n \geq n_0\}$ is dense in $J_{0+}$ and 
$\{b_{n, k}\; |\; n, k\in\mathbb{N}, n\geq n_0\}$ is dense in $B_{0+}$. 
The same proof as in \cite[Section 5]{Na4} shows that $J_0$ is stable. 
Since we consider $(B^{\omega}\cap B^{\prime})^{\beta}$ rather than $(B^{\omega})^{\beta}$, 
we shall give a detailed proof of the following lemma. 

\begin{lem}
With notation as above, $\eta_0$ is a separable purely large extension. 
\end{lem}
\begin{proof}
Let $x\in B_0\setminus J_0$. It suffices to show that $\overline{xJ_0x^*}=
\overline{xx^*J_0xx^*}$ contains a stable C$^*$-subalgebra which is full in $J_0$. 
For any $l\in\mathbb{N}$, 
there exist natural numbers $n(l)$ and $k(l)$ such that 
$$
\| b_{n(l),k(l)}-xx^* \| < \dfrac{1}{l}
$$ 
since $\{b_{n, k}\; |\; n, k\in\mathbb{N}\}$ is dense in $B_{0+}$. 
Note that we have 
$$
\|  (b_{n(l),k(l)}-1/l)_{+}-xx^*\| < \dfrac{2}{l}\to 0 \quad \text{as}\quad l\to\infty .
$$
Therefore there exists a natural number $l_0$ such that 
$(b_{n(l_0),k(l_0)}-1/l_0)_{+}\notin J$, that is, $(k(l_0), l_0)\in T_{n(l_0)}$ because 
$J$ is a closed ideal and $xx^*\notin J$. 
Note that $(k(l_0), l_0)\in T_{n(l_0)}$ implies that 
$\overline{(b_{n(l_0),k(l_0)}-1/l_0)_{+}J_0(b_{n(l_0),k(l_0)}-1/l_0)_{+}}$ is full in $J_0$. 
By the inequality above and \cite[Lemma 2.2]{KR2}, there exists a contraction $r$ in $B_0$ such that 
$r^*xx^*r=(b_{n(l_0),k(l_0)}-1/l_0)_{+}$. Put $y:= (xx^*)^{1/2}r$, and 
let $y=v|y|$ be the polar decomposition of $y$ in $B_0^{**}$. 
Note that if $b_1, b_2\in \overline{y^*B_0y}$, then $vb_1 \in B_0$ and $b_1v^*vb_2=b_1b_2$ by 
\cite[III. 5.2.16]{Bla}. Furthermore, it is easy to see that we have 
$$
v\overline{(b_{n(l_0),k(l_0)}-1/l_0)_{+}B_0(b_{n(l_0),k(l_0)}-1/l_0)_{+}}v^*=v\overline{y^*B_0y}v^*= 
\overline{yBy^*}\subset \overline{xx^*B_0xx^*}.
$$ 
Set 
$$
C:= v\overline{(b_{n(l_0),k(l_0)}-1/l_0)_{+}J_0(b_{n(l_0),k(l_0)}-1/l_0)_{+}}v^*\subset \overline{xx^*J_0xx^*}.
$$
Then it can be easily checked that $C$ is full in $J_0$. We shall show that 
$C$ is stable. 
Let $a\in C_{+}\setminus\{0\}$ and $\varepsilon>0$, and put 
$$
\varepsilon^{\prime}:= \min\left\{\varepsilon,\; \dfrac{\varepsilon}{2\| a\|},\; \sqrt{\dfrac{\varepsilon}{2}}
\right\}.
$$ 
Since $\{a_{n, m}\; |\; n, m\in\mathbb{N}, n \geq n(l_0)\}$ is dense in $J_{0+}$, there exist  
$n_0\geq n(l_0)$ and $m_0\in\mathbb{N}$ such that 
$$
\| a -v(b_{n(l_0),k(l_0)}-1/l_0)_{+}a_{n_0, m_0}(b_{n(l_0),k(l_0)}-1/l_0)_{+}v^*\| < \varepsilon^{\prime}\leq 
\varepsilon.
$$
Let $a^{\prime}:=(b_{n(l_0),k(l_0)}-1/l_0)_{+}a_{n_0, m_0}(b_{n(l_0),k(l_0)}-1/l_0)_{+}$ 
and $a^{\prime\prime}:=va^{\prime}v^*$. Note that $a^{\prime\prime}\in C$ and 
$\|a -a^{\prime\prime}\|< \varepsilon$. 
By construction of $B_0$ and $J_0$, there exist elements 
$$
c=c_{n_0,n(l_0),(k(l_0),l_0),m_0}\quad \text{and}\quad z=z_{n_0,n(l_0),(k(l_0),l_0),m_0} 
$$
in  $\overline{(b_{n(l_0),k(l_0)}-1/l_0)_{+}J_0(b_{n(l_0),k(l_0)}-1/l_0)_{+}}$
such that 
$$
a^{\prime}c=0, \quad zz^*=c \quad \text{and} \quad z^*z
=a^{\prime}.
$$
Set $z^{\prime}:= vzv^*$ and $c^{\prime}:=vcv^*$, then $z^{\prime}$ and $c^{\prime}$ are elements in 
$C$. 
Since we have 
$$
z^{\prime*}z^{\prime}=vz^*v^*vzv^*=vz^*zv^*=va^{\prime}v^*=a^{\prime\prime},
$$
$$
z^{\prime}z^{\prime*}=vzv^{*}vz^*v=vzz^*v^*=vcv^*=c^{\prime}
$$
and
$$
a^{\prime\prime}c^{\prime}=va^{\prime}v^*vcv^*=va^{\prime}cv^*=0,
$$
$a^{\prime\prime}$ is Murray-von Neumann equivalent to $c^{\prime}$ in $C$ and 
$a^{\prime\prime}c^{\prime}$=0. 
Furthermore, 
\begin{align*}
\| ac^{\prime}\| =\| ac^{\prime}-a^{\prime\prime}c^{\prime}\|\leq  \|a-a^{\prime\prime}\| \|c^{\prime}\| 
< \varepsilon^{\prime} \| a^{\prime\prime}\|
< \varepsilon^{\prime}(\|a\| +\varepsilon^{\prime})\leq \varepsilon.
\end{align*}
Hence $C$ is stable by \cite[Theorem 2.2]{Ror2} (see also \cite{HR}). 
Consequently, we obtain the conclusion. 
\end{proof}

Using the lemma above, the same proof as \cite[Lemma 5.2]{Na4} shows the following theorem.

\begin{thm}
Let $B$ be a simple separable non-type I nuclear monotracial C$^*$-algebra with strict comparison, 
and let $\beta$ be an outer action of a countable discrete amenable group 
$\Gamma$ on $B$. Then there exists a homomorphism $\Phi$ from $\mathcal{W}$ to 
$(B^{\omega}\cap B^{\prime})^{\beta}$ such that $\tau_{B, \omega}\circ \Phi=\tau_{\mathcal{W}}$. 
\end{thm}

The following corollary is the main theorem in this section. 

\begin{cor}
Let $A$ be a simple separable nuclear monotracial C$^*$-algebra and $B$ 
a simple separable non-type I nuclear monotracial C$^*$-algebra with strict comparison 
and $B\subset\overline{\mathrm{GL}(B^{\sim})}$, 
and let $\alpha$ be a strongly outer action of a countable discrete amenable group 
$\Gamma$ on $A$ and $\beta$ an outer action of $\Gamma$ on $B$. 
Assume that $A\rtimes_{\alpha}\Gamma$ is $\mathcal{W}$-embeddable. 
Then there exists a trace preserving sequential asymptotic cocycle morphism from $(A, \alpha)$ 
to $(B, \beta)$. 
\end{cor}
\begin{proof}
Note that $A\rtimes_{\alpha}\Gamma$ is monotracial because $\alpha$ is strongly outer. 
By \cite[Lemma 6.3]{Na4}, 
there exists a trace preserving homomorphism from  $A\rtimes_{\alpha}\Gamma$ to $\mathcal{W}$ 
since $A\rtimes_{\alpha}\Gamma$ is $\mathcal{W}$-embeddable. 
Hence the theorem above implies that there exists a homomorphisms $\Phi$ from 
$A\rtimes_{\alpha}\Gamma$ to $(B^{\omega}\cap B^{\prime})^{\beta}\subset (B^{\omega})^{\beta}$ 
such that $\tau_{A\rtimes_{\alpha}\Gamma}=\tau_{B, \omega}\circ \Phi$. 
Therefore we obtain the conclusion by Lemma \ref{lem:exist}. 
\end{proof}

The following corollary is an immediate consequence of the corollary above. 
Note that for any unitary element $U$ in $(B^{\sim})^{\omega}$, there exists a sequence 
$\{u_n\}_{n\in\mathbb{N}}$ of unitary elements in $B^{\sim}$ such that 
$U=(u_n)_n$ in $(B^{\sim})^{\omega}$ because we assume $B\subset 
\overline{\mathrm{GL}(B^{\sim})}$. 

\begin{cor}\label{cor:main-section7}
Let $A$ be a simple separable nuclear monotracial C$^*$-algebra and $B$ 
a simple separable non-type I nuclear monotracial C$^*$-algebra with strict comparison 
and $B\subset\overline{\mathrm{GL}(B^{\sim})}$, 
and let $\alpha$ be a strongly outer action of a countable discrete amenable group 
$\Gamma$ on $A$ and $\beta$ an outer action of $\Gamma$ on $B$. 
Assume that $A\rtimes_{\alpha}\Gamma$ is $\mathcal{W}$-embeddable. 
For any finite subsets $F\subset A$, $\Gamma_0\subset \Gamma$ and $\varepsilon>0$, 
there exists a proper $(\Gamma_0, F, \varepsilon)$-approximate cocycle morphism $(\varphi, u)$ 
from $(A, \alpha)$ to $(B, \beta)$ such that 
$$
| \tau_{B} (\varphi (a)) -\tau_{A} (a) | < \varepsilon 
$$
for any $a\in F$. 
\end{cor}

\section{Characterization and cocycle conjugacy}\label{sec:main}

In this section we shall show the main results in this paper. 

Recall that 
for any countable discrete group $\Gamma$,  $\mu^{\Gamma}$ is the Bernoulli shift action  
of $\Gamma$ on $\bigotimes_{g\in \Gamma}M_{2^{\infty}}\cong M_{2^{\infty}}$. 
It is known that $\mu^{\Gamma}$ is strongly outer (see, for example \cite[Lemma 2.5]{VV}). 
Hence if $\Gamma$ is amenable, then 
$M_{2^{\infty}}\rtimes_{\mu^{\Gamma}}\Gamma$ is a simple separable nuclear 
monotracial C$^*$-algebra by \cite[Theorem 3.1]{K} and \cite[Proposition 2.1]{Na0}. 
Therefore $(M_{2^{\infty}}\otimes\mathcal{W})
\rtimes_{\mu^{\Gamma}\otimes\mathrm{id}_{\mathcal{W}}}\Gamma$ 
is isomorphic to $\mathcal{W}$ by \cite{CE} and \cite{EGLN} (or \cite[Corollary 6.2]{Na4}) since 
$(M_{2^{\infty}}\otimes\mathcal{W})\rtimes_{\mu^{\Gamma}\otimes\mathrm{id}_{\mathcal{W}}}
\Gamma \cong(M_{2^{\infty}}\rtimes_{\mu^{\Gamma}}\Gamma)\otimes\mathcal{W}$. 
The following theorem is one of the main results in this paper. 

\begin{thm}\label{thm:main2}
Let $\gamma$ be a strongly outer action of a countable discrete amenable group $\Gamma$ on 
$\mathcal{W}$. Then $\gamma$ is cocycle conjugate to 
$\mu^{\Gamma}\otimes\mathrm{id}_{\mathcal{W}}$ on $M_{2^{\infty}}\otimes \mathcal{W}$ 
if and only if $(\mathcal{W}, \gamma)$ satisfies the following properties: \ \\
(i) for any $\theta\in [0,1]$, there exists a projection $p$ in 
$F(\mathcal{W})^{\gamma}$ such that $\tau_{\mathcal{W}, \omega}(p)=\theta$, \ \\
(ii) if $p$ and $q$ are projections in $F(\mathcal{W})^{\gamma}$ 
such that $0<\tau_{\mathcal{W}, \omega}(p)=\tau_{\mathcal{W}, \omega}(q)$, 
then  $p$ is Murray-von Neumann equivalent to $q$, \ \\
(iii) $\mathcal{W}\rtimes_{\gamma}\Gamma$ is $\mathcal{W}$-embeddable.  
\end{thm}
\begin{proof}
First, we shall show the only if part. 
Since $\gamma$ is cocycle conjugate to 
$\mu^{\Gamma}\otimes\mathrm{id}_{\mathcal{W}}$, 
$\mathcal{W}\rtimes_{\gamma}\Gamma$ is isomorphic to 
$(M_{2^{\infty}}\otimes\mathcal{W})\rtimes_{\mu^{\Gamma}\otimes\mathrm{id}_{\mathcal{W}}}\Gamma
\cong \mathcal{W}$. 
We see that $F(\mathcal{W})^{\gamma}$ is isomorphic to 
$F(M_{2^{\infty}}\otimes\mathcal{W})^{\mu^{\Gamma}\otimes\mathrm{id}_{\mathcal{W}}}$ because 
$\gamma$ on $F(\mathcal{W})$ is conjugate to $\mu^{\Gamma}\otimes\mathrm{id}_{\mathcal{W}}$ 
on $F(M_{2^{\infty}}\otimes\mathcal{W})$. 
Therefore $F(\mathcal{W})^{\gamma}$ satisfies the properties (i) and (ii) 
by Proposition \ref{pro:equivariant-w} and Corollary \ref{cor:main-section4}. 
Consequently, the proof of the only if part is complete. 

We shall show the if part. 
As above, $F(M_{2^{\infty}}\otimes\mathcal{W})^{\mu^{\Gamma}\otimes\mathrm{id}_{\mathcal{W}}}$ 
also satisfies the properties (i) and (ii). 
Therefore $(\mathcal{W}, \gamma)$ and $(M_{2^{\infty}}\otimes\mathcal{W}, 
\mu^{\Gamma}\otimes\mathrm{id}_{\mathcal{W}})$ satisfy the assumption in 
Corollary \ref{cor:main-section6} (uniqueness) and Corollary \ref{cor:main-section7} (existence). 
For notational simplicity, put $(A, \alpha):= (\mathcal{W}, \gamma)$ and 
$(B, \beta):=(M_{2^{\infty}}\otimes\mathcal{W}, \mu^{\Gamma}\otimes\mathrm{id}_{\mathcal{W}})$.  
We will use Proposition \ref{pro:intertwining} (the approximate intertwining) to obtain the conclusion. 

Let $\{x_n\; |\; n\in\mathbb{N}\}$ be a dense subset of $A$ and 
$\{y_n\; |\; n\in\mathbb{N}\}$ a dense subset of $B$. 
Since $\Gamma$ is countable, there exists a increasing sequence of inverse closed finite subsets 
$\{\Gamma_n\}_{n\in\mathbb{N}}$ of $\Gamma$ such that 
$\bigcup_{n=1}^{\infty}\Gamma_n=\Gamma$. 
For any $n\in\mathbb{N}$, 
let $\varepsilon_n:=1/2^n$. 
Note that $\iota \in\Gamma_1$ and $\sum_{n=1}^\infty \varepsilon_n < \infty$. 
Applying Corollary \ref{cor:main-section6} to $F_1:=\{x_1, x_1^*\}$, $\Gamma_1$ 
and $\varepsilon_1$, 
we obtain $F_1^{\prime}\subset A$, $\Gamma_1^{\prime}\subset \Gamma$ and $\delta_1>0$. 
We may assume $\delta_1<\varepsilon_1$. 
Put 
$$
K_1:= \{gh\; |\; g,h\in\Gamma_1 
\cup \Gamma_1^{\prime}\} \quad \text{and} \quad F_1^{\prime\prime}:= F_1\cup F_1^{\prime}. 
\eqno{(8.1.1)}
$$
By Corollary \ref{cor:main-section7}, 
there exists a proper approximate 
cocycle morphism $(\varphi_1, u^{(1)})$ from $(A, \alpha)$ to $(B, \beta)$ such that 
$$
(\varphi_1, u^{(1)}) \text{ is  a proper } (K_1, F_1^{\prime\prime}, \delta_1/2) 
\text{-approximate cocycle morphism} \eqno{(8.1.2)}
$$
and 
$$
| \tau_{B}(\varphi_1(a)) -\tau_{A}(a)| < \frac{\delta_1}{2} \eqno{\mathrm{(8.1.3)}}
$$
for any $a\in F_1^{\prime\prime}$. 
Applying Corollary \ref{cor:main-section6} to $G_1:=\{y_1, y_1^*\}$, $\Gamma_1$ 
and $\varepsilon_1$, 
we obtain $G_1^{\prime}\subset B$, $\Gamma_1^{\prime\prime}\subset \Gamma$ and 
$\delta_2>0$. 
We may assume  
$$
\delta_2 < \frac{\delta_1}{4+\displaystyle{\max_{a\in F_1^{\prime}}\|a\|}}. \eqno{(8.1.4)}
$$
Note that we have $\delta_2<\frac{\delta_1}{4}<\delta_1<\varepsilon_1$. 
Put 
$$
K_2:=\{gh \;|\; g, h\in K_1\cup\Gamma_1^{\prime\prime}\},
$$ 
and choose a finite subset $G_1^{\prime\prime}$ of $B$ such that 
$$
G_1\cup G_1^{\prime}\cup \beta_g(\varphi_1(F_1^{\prime\prime}))u_{g}^{(1)} \subset G_1^{\prime\prime}
\;
\text{and}
\;
\beta_g(u_h^{(1)})\in \{b+\lambda 1_{B^{\sim}}\; |\; b\in G_1^{\prime\prime}, \lambda\in\mathbb{C}\}
\eqno{(8.1.5)}
$$
for any $g,h\in K_2$. Note that we have $\iota\in K_2$, $u_{\iota}=1_{B^{\sim}}$ and 
$\beta_{\iota}=\mathrm{id}_{B}$. 
By Corollary \ref{cor:main-section7}, there exists a proper approximate 
cocycle morphism $(\psi_1, v^{(1)})$ from $(B, \beta)$ to $(A, \alpha)$ such that 
$$
(\psi_1, v^{(1)}) \text{ is  a proper } (K_2, G_1^{\prime\prime}, \delta_2/2)
\text{-approximate cocycle morphism} \eqno{(8.1.6)}
$$
and 
$$
| \tau_{A}(\psi_1(b))- \tau_{B}(b) | < \frac{\delta_2}{2} \eqno{(8.1.7)}
$$ 
for any $b\in G_1^{\prime\prime}$. 
By (8.1.1), (8.1.3),  (8.1.5), (8.1.6) and (8.1.7), 
$\psi_1\circ \varphi_1$ is an $(F_1^{\prime}, \delta_1)$-multiplicative map 
such that 
$$
| \tau_{A}(\psi_1\circ \varphi_1(a))-\tau_{A}(a) | < \delta_1
$$
for any $a\in F_1^{\prime}$. For any $g\in\Gamma$, set 
$$
z_g^{(1)}:= \psi_1^{\sim}(u_g^{(1)})v_{g}^{(1)} \in A^{\sim}.
$$  
Then for any $a\in F_1^{\prime}\subset F_1^{\prime\prime}$ and 
$g, h\in \Gamma_1^{\prime}\subset K_1\subset K_2$, we have 
\begin{align*}
 \| \psi_1\circ \varphi_1(a)(z^{(1)}_{gh} &-z^{(1)}_g\alpha_g(z_h^{(1)}))\| \\
&= \| \psi_1(\varphi_1(a))\psi_1^{\sim}(u_{gh}^{(1)})v_{gh}^{(1)} 
-\psi_1(\varphi_1(a))z^{(1)}_g\alpha_g(z_h^{(1)})\| \\
(8.1.5), (8.1.6)\; 
& <\| \psi_1(\varphi_1(a)u_{gh}^{(1)})v_{gh}^{(1)}-\psi_1(\varphi_1(a)) 
z^{(1)}_g\alpha_g(z_h^{(1)})\| + \frac{\delta_2}{2} \\
(8.1.5), (8.1.6)\;
& < \|\psi_1(\varphi_1(a)u_{gh}^{(1)})v_{g}^{(1)}\alpha_g(v_{h}^{(1)})- \psi_1(\varphi_1(a)) 
z^{(1)}_g\alpha_g(z_h^{(1)})\| +\delta_2 \\
(8.1.2)\; 
& <\| \psi_1(\varphi_1(a)u_{g}^{(1)}\beta_g(u_{h}^{(1)}))v_{g}^{(1)}\alpha_g(v_{h}^{(1)}) 
- \psi_1(\varphi_1(a)) 
z^{(1)}_g\alpha_g(z_h^{(1)})\| \\ 
&\quad +\delta_2+\frac{\delta_1}{2} \\
(8.1.5), (8.1.6)\;
& <\| \psi_1(\varphi_1(a)u_{g}^{(1)})\psi_1^{\sim}(\beta_g(u_{h}^{(1)}))v_{g}^{(1)}\alpha_g(v_{h}^{(1)})  
\\
&\quad -\psi_1(\varphi_1(a))z^{(1)}_g\alpha_g(z_h^{(1)}) \| 
 +\frac{3\delta_2}{2}+\frac{\delta_1}{2} \\
(8.1.5), (8.1.6)\;
& <\| \psi_1(\varphi_1(a)u_{g}^{(1)})v_{g}^{(1)}\alpha_g(\psi_1^{\sim}(u_{h}^{(1)}))v_{g}^{(1)*}
v_{g}^{(1)}\alpha_g(v_{h}^{(1)}) \\
&\quad - \psi_1(\varphi_1(a)) z^{(1)}_g\alpha_g(z_h^{(1)})\| 
 +\frac{(3+\| a\|)\delta_2}{2}+\frac{\delta_1}{2} \\ 
& = \| (\psi_1(\varphi_1(a)u_{g}^{(1)})v_{g}^{(1)}\alpha_g(z_h^{(1)})
 - \psi_1(\varphi_1(a)) z^{(1)}_g\alpha_g(z_h^{(1)})\| \\
& \quad +\frac{(3+\| a\|)\delta_2}{2}+\frac{\delta_1}{2} \\ 
(8.1.5), (8.1.6)\;
& < \frac{(4+\| a\|)\delta_2}{2}+\frac{\delta_1}{2}  \\
(8.1.4)\;
& < \delta_1,
\end{align*}
\begin{align*}
\|z_g^{(1)*}z_g^{(1)}-1 \| 
&=\| v_{g}^{(1)*}\psi_1^{\sim}(u_g^{(1)*})\psi_1^{\sim}(u_g^{(1)})v_{g}^{(1)}-1\|
=\|\psi_1^{\sim}(u_g^{(1)*})\psi_1^{\sim}(u_g^{(1)})-1\| \\
(8.1.5), (8.1.6)\;
&<\frac{\delta_2}{2}<\delta_1,
\end{align*}
\begin{align*}
\|z_g^{(1)}z_g^{(1)*}-1 \| 
&=\|\psi_1^{\sim}(u_g^{(1)})\psi_1^{\sim}(u_g^{(1)*})-1\| \\
(8.1.5), (8.1.6)\;
&<\frac{\delta_2}{2}<\delta_1
\end{align*}
and 
\begin{align*}
&\| (\psi_1\circ \varphi_1)\circ \alpha_g(a)- \mathrm{Ad}(z_g^{(1)})\circ 
\alpha_g\circ (\psi_1\circ \varphi_1)(a)\| \\
&= \| \psi_1(\varphi_1(\alpha_g(a)))-\psi_1^{\sim}(u_g^{(1)})v_{g}^{(1)}\alpha_g(\psi_1(\varphi_1(a)))
v_g^{(1)*}\psi_1^{\sim}(u_g^{(1)*})\| \\
(8.1.5), (8.1.6)\;
&< \| \psi_1(\varphi_1(\alpha_g(a)))-\psi_1^{\sim}(u_g^{(1)})\psi_1(\beta_g(\varphi_1(a)))
\psi_1^{\sim}(u_g^{(1)*})\| +\frac{\delta_2}{2} \\
(8.1.5), (8.1.6)\;
&< \| \psi_1(\varphi_1(\alpha_g(a)))-\psi_1(u_g^{(1)}\beta_g(\varphi_1(a))u_g^{(1)*})\| +\frac{3\delta_2}{2}
\\
(8.1.2)\; 
& <\frac{\delta_1}{2}+\frac{3\delta_2}{2}<\delta_1.
\end{align*}
Therefore $(\psi_1\circ \varphi_1, z^{(1)})$ is a proper $(\Gamma_1^{\prime}, F_1^{\prime}, \delta_1)$-
approximate quasi cocycle morphism from $(A, \alpha)$ to $(A, \alpha)$. 
Hence  Corollary \ref{cor:main-section6} implies that there exists a unitary element $w_1$ in 
$A^{\sim}$ such that 
$$
\| \psi_1\circ \varphi_1(a)- w_1aw_1^* \| < \varepsilon_1
\quad \text{and} \quad  
\| \psi_1\circ \varphi_1(a)(z_g^{(1)}- w_1\alpha_g(w_1^*))\| < \varepsilon_1
$$
for any $a\in F_1$ and $g\in \Gamma_1$. Put $(\varphi_1^{\prime}, \tilde{u}^{(1)})
:=(\varphi_1, u^{(1)})$ and 
$$ 
\psi_1^{\prime}:= \mathrm{Ad}(w_1^*)\circ \psi_1 \quad \text{and}
\quad \tilde{v}_g^{(1)}:= w_1^*v_{g}^{(1)}\alpha_g(w_1)
$$
for any $g\in\Gamma$. It is easy to see that $(\psi_1^{\prime}, \tilde{v}^{(1)})$ is 
a proper 
$(K_2, G_1^{\prime\prime}, \delta_2/2)$-approximate 
cocycle morphism from $(B, \beta)$ to $(A, \alpha)$ such that 
$$
| \tau_{A}(\psi^{\prime}_1(b))- \tau_{B}(b) | < \frac{\delta_2}{2}
$$ 
for any $b\in G_1^{\prime\prime}$.
Furthermore, we have 
$$
\| \psi_1^{\prime}\circ \varphi^{\prime}_1(a) - a\| <\varepsilon_1
$$
and 
\begin{align*}
\| \psi_1^{\prime}\circ\varphi_1^{\prime}(a)(\psi_1^{\prime\sim}(\tilde{u}_g^{(1)})\tilde{v}_g^{(1)}-1)\| 
&= \|\psi_1\circ \varphi_1(a)z_g^{(1)}\alpha_g(w_1)-\psi_1\circ \varphi_1(a)w_1\|< \varepsilon_1
\end{align*}
for any $a\in F_1$ and $g\in \Gamma_1$. 

In a similar way as above, 
applying Corollary \ref{cor:main-section6} to $F_2:=F_1\cup \{x_2, x_2^*\}$, $\Gamma_2$ 
and $\varepsilon_2$, 
we obtain a finite subset $F_2^{\prime}$ of $A$, 
a finite subset $\Gamma_2^{\prime}$ of $\Gamma$ and a strictly positive number 
$\delta_3$ satisfying  
$$
\delta_3 
< \min\left\{\frac{\delta_2}{4+\displaystyle{\max_{b\in G_{1}^{\prime}}\|b\|}},\; \varepsilon_2\right\}.
$$
Put 
$$
K_3:=\{gh \;|\; g, h\in K_2\cup\Gamma_2^{\prime}\},
$$ 
and choose a finite subset $F_2^{\prime\prime}$ of $A$ such that 
$$
F_2\cup F_2^{\prime}\cup \alpha_g(\psi^{\prime}_1(G_1^{\prime\prime}))\tilde{v}_{g}^{(1)} \subset 
F_2^{\prime\prime}
\quad
\text{and} 
\quad
\alpha_g(\tilde{v}_h^{(1)})\in 
\{a+\lambda 1_{A^{\sim}}\; |\; a\in F_2^{\prime\prime}, \lambda\in\mathbb{C}\}
$$
for any $g,h\in K_3$. 
By Corollary \ref{cor:main-section7}, there exists a proper 
$(K_3, F_2^{\prime\prime}, \delta_3/2)$-approximate 
cocycle morphism $(\varphi_2, u^{(2)})$ from $(A, \alpha)$ to $(B, \beta)$ such that 
$$
| \tau_{B}(\varphi_2(a))- \tau_{A}(a) | < \frac{\delta_3}{2}
$$ 
for any $b\in F_2^{\prime\prime}$. 
For any $g\in\Gamma$, set 
$$
z_g^{(2)}:= \varphi_2^{\sim}(\tilde{v}_g^{(1)})u_{g}^{(2)} \in B^{\sim}.
$$  
By a similar argument as above, we see that there exists a unitary element $w_2$ in 
$B^{\sim}$ such that 
$$
\| \varphi_2\circ \psi^{\prime}_1(b)- w_2bw_2^* \| < \varepsilon_1
\quad \text{and} \quad  
\| \varphi_2\circ \psi_1^{\prime}(b)(z_g^{(2)}- w_2\beta_g(w_2^*))\| < \varepsilon_1
$$
for any $b\in G_1$ and $g\in \Gamma_1$.
Put 
$$ 
\varphi_2^{\prime}:= \mathrm{Ad}(w_2^*)\circ \varphi_2 \quad \text{and}
\quad \tilde{u}_g^{(2)}:= w_2^*u_{g}^{(2)}\beta_g(w_2)
$$
for any $g\in\Gamma$. Then 
a similar argument as above shows that 
$(\varphi_2^{\prime}, \tilde{u}^{(2)})$ is 
a proper 
$(K_3, F_2^{\prime\prime}, \delta_3/2)$-approximate 
cocycle morphism from $(A, \alpha)$ to $(B, \beta)$ such that 
$$
| \tau_{B}(\varphi^{\prime}_2(a))- \tau_{A}(a) | < \frac{\delta_3}{2}
$$ 
for any $a\in F_2^{\prime\prime}$, and we have 
$$
\| \varphi_2^{\prime}\circ \psi^{\prime}_1(b) - b\| <\varepsilon_1 \quad 
\text{and} \quad  
\| \varphi_2^{\prime}\circ\psi_1^{\prime}(b)(\varphi_2^{\prime\sim}(\tilde{v}_g^{(1)})\tilde{u}_g^{(2)}-1)\| 
< \varepsilon_1
$$
for any $b\in G_1$ and $g\in \Gamma_1$. Put $G_2:=G_1\cup \{y_2, y_2^*\}$. 
Repeating this process, we obtain 
sequences of proper approximate cocycle morphisms 
$\{(\varphi^{\prime}_n, \tilde{u}_n)\}_{n=1}^{\infty}$ from $(A, \alpha)$ to $(B, \beta)$, 
$\{(\psi^{\prime}_n, \tilde{v}_n)\}_{n=1}^{\infty}$ from $(B,\beta)$ to $(A, \alpha)$ 
and increasing sequences of finite self-adjoint subsets $\{F_n\}_{n=1}^\infty$ of $A$, 
$\{G_n\}_{n=1}^{\infty}$ of $B$ satisfying the assumption of 
Proposition \ref{pro:intertwining}. 
Consequently, $\alpha$ is cocycle conjugate to $\beta$ by 
Proposition \ref{pro:intertwining}. 
\end{proof}

The following corollary is one of the main results in this paper. 

\begin{cor}\label{cor:main1}
Let $A$ and $B$ be simple separable nuclear monotracial C$^*$-algebras, and let $\alpha$ and $\beta$ 
be strongly outer actions of a countable discrete amenable group $\Gamma$ on 
$A$ and $B$, respectively. Then $\alpha\otimes\mathrm{id}_{\mathcal{W}}$ on $A\otimes\mathcal{W}$ 
is cocycle conjugate to $\beta\otimes\mathrm{id}_{\mathcal{W}}$ on $B\otimes\mathcal{W}$.
\end{cor}
\begin{proof}
Note that $(A\otimes\mathcal{W})\rtimes_{\alpha\otimes\mathrm{id}_{\mathcal{W}}}\Gamma$ and 
$(B\otimes\mathcal{W})\rtimes_{\beta\otimes\mathrm{id}_{\mathcal{W}}}\Gamma$ are isomorphic to 
$\mathcal{W}$ by the same reason as 
$(M_{2^{\infty}}\otimes\mathcal{W})\rtimes_{\mu^{\Gamma}\otimes\mathrm{id}_{\mathcal{W}}}\Gamma$. 
Therefore we obtain the conclusion by Proposition \ref{pro:equivariant-w}, 
Corollary \ref{cor:main-section4} and the theorem above. 
\end{proof}

\begin{rem}
(1) There exist uncountably many non-cocycle conjugate strongly outer actions of 
$\mathbb{Z}_2$ on $\mathcal{W}$ by \cite[Example 5.6]{Na0} and \cite[Remark 5.7]{Na0}. 
\ \\
(2) We cannot expect an analog of Suzuki's generalization in \cite{Su} for $\mathcal{W}$. 
Note that no action of a countable discrete non-amenable exact group 
on $\mathcal{W}$ is amenable (or QAP) because $\mathcal{W}$ is monotracial. Indeed, 
if $\alpha$ is an amenable action of a countable discrete exact 
group $\Gamma$ on $\mathcal{W}$, then $\alpha$ induces an amenable action on 
$\{\tau_{\mathcal{W}}\}$. This implies that $\Gamma$ is amenable. 
We refer the reader to \cite{OS} and references given there for details of amenable actions and 
QAP actions. 
\end{rem}

\end{document}